\RequirePackage{fix-cm}
\documentclass[smallextended]{svjour3}       
\smartqed  
\usepackage{graphicx}
\usepackage{epstopdf}

\usepackage{import}
\usepackage{amsmath}
\usepackage{amssymb}
\usepackage[alphabetic, initials]{amsrefs}
\usepackage[english]{babel}
\usepackage{url}
\usepackage{fancyhdr}
\usepackage{graphicx}
\usepackage[utf8]{inputenc}
\usepackage[
colorlinks=true,
linkcolor=black, 
anchorcolor=black,
citecolor=black, 
urlcolor=black, 
]{hyperref}
\usepackage{geometry}
\usepackage[all]{xy}

\newcommand{\bigslant}[2]{{\raisebox{.2em}{$#1$}\left/\raisebox{-.2em}{$#2$}\right.}}

\newcommand{\leftexp}[2]{{\vphantom{#2}}^{#1}{#2}}
\newcommand{\leftexpsub}[3]{{\vphantom{#3}}^{#1}_{#2}{#3}}
\providecommand{\abs}[1]{\left\lvert#1\right\rvert}

\newcommand{\op}[1]{\operatorname{#1}}
\newcommand{\oop}{\operatorname{op}}
\newcommand{\cop}{\operatorname{cop}}
\newcommand{\ov}[1]{\overline{#1}}
\newcommand{\un}[1]{\underline{#1}}
\newcommand{\lmod}[1]{#1\text{-}\mathbf{Mod}}
\newcommand{\rmod}[1]{\mathbf{Mod}\text{-}#1}

\newcommand{\YDasy}[3]{\leftexp{#1}{\cal{YD}}^{#2}_{\op{asy}}(#3)}

\newcommand{\ad}{\operatorname{ad}}

\newcommand{\coev}{\operatorname{coev}}

\newcommand{\Drin}{\operatorname{Drin}}
\newcommand{\aDrin}{\operatorname{Drin}^{\operatorname{asy}}}

\newcommand{\Hom}{\operatorname{Hom}}
\newcommand{\ide}{\operatorname{Id}}

\newcommand{\ord}{\operatorname{ord}}

\newcommand{\Wor}{\operatorname{Wor}}

\newcommand{\HYD}{\leftexpsub{H}{H}{\cal{YD}}}

\newcommand{\Vect}{\mathbf{Vect}_k}

\newcommand{\g}{\mathfrak{g}}

\newcommand{\Ug}{U_q(\mathfrak{g})}
\newcommand{\Uq}{U_q(\mathfrak{sl}_2)}

\providecommand{\cal}[1]{\mathcal{#1}}

\providecommand{\fr}[1]{\mathfrak{#1}}

\providecommand{\op}[1]{\operatorname{#1}}

\newcommand{\mC}{\mathbb{C}}

\newcommand{\mZ}{\mathbb{Z}}

\newcommand{\cD}{\mathcal{D}}
\newcommand{\cB}{\mathcal{B}}

\newcommand{\cI}{\mathcal{I}}

\begin{document}

\title{Pointed Hopf Algebras with Triangular Decomposition\thanks{Supported by EPSRC grant EP/I033343/1}
}
\subtitle{A Characterization of Multiparameter Quantum Groups}

\titlerunning{Pointed Hopf Algebras with Triangular Decomposition}        

\author{Robert Laugwitz        
}


\institute{School of Mathematics, University of East Anglia, Norwich Research Park,
Norwich,
UK,
 NR4 7TJ\\
   \email{r.laugwitz@uea.ac.uk}           
}

\date{}

\maketitle

\begin{abstract}
In this paper, we present an approach to the definition of multiparameter quantum groups by studying Hopf algebras with triangular decomposition. Classifying all of these Hopf algebras which are of what we call weakly separable type over a group, we obtain a class of pointed Hopf algebras which can be viewed as natural generalizations of multiparameter deformations of universal enveloping algebras of Lie algebras. These Hopf algebras are instances of a new version of braided Drinfeld doubles, which we call \emph{asymmetric} braided Drinfeld doubles. This is a generalization of an earlier result by Benkart and Witherspoon (2004) who showed that two-parameter quantum groups are Drinfeld doubles. It is possible to recover a Lie algebra from these doubles in the case where the group is free abelian and the parameters are generic. The Lie algebras arising are generated by Lie subalgebras isomorphic to $\mathfrak{sl}_2$.
\end{abstract}


\section{Introduction}\label{section1}

\subsection{What Are Quantum Groups?}\label{motivation}
An important problem in the theory of quantum groups is to give some definition of a class of these objects that captures known series of quantum groups, such as the quantum enveloping algebras $U_q(\fr{g})$ of \cite{Dri}, and their finite-dimensional analogues, as examples. This was for example formulated in \cite[Problem II.10.2]{BG}:

\begin{quotation}
\begin{it}
``Given a finite-dimensional Lie algebra $\g$, find axioms for Hopf al\-ge\-bras to qualify as quantized enveloping algebras of this particular $\g$."
\end{it}
\end{quotation}

A possible hint to the structure of quantum groups is that the quantum envel\-oping algebras $\Ug$ (as well as the small quantum groups $u_q(\fr{g})$ and multiparameter versions) are \emph{pointed Hopf algebras}. Such Hopf algebras were studied by several authors (see e.g. \cite{AS}). Classification results as in \cite{AS2} suggest a strong resemblance of all finite-dimensional pointed Hopf algebras over abelian groups with small quantum groups. Another paper \cite{AS3} gives a characterization of quantum groups at generic pa\-ram\-e\-ters using pointed Hopf algebras of finite Gelfand--Kirillov dimension with infinitesimal braiding of positive generic type.

A further hint to the structure of quantum groups is that they can be decomposed in a triangular way (via the PBW theorem) as
\[
\Ug=U_q(\fr{n}_+)\otimes k\mZ^n\otimes U_q(\fr{n}_-).
\]
Here, the positive and negative part are perfectly paired braided Hopf algebras, and the relation with the group algebra $k\mZ^n$ is governed by semidirect product relations. The positive (and negative) part are so-called \emph{Nichols algebras}.

A third aspect --- observed already in the original paper \cite{Dri} --- is that quantum groups are (quotients of) \emph{quantum} or \emph{Drinfeld doubles}. It was shown in \cite{Maj2} that $U_q(\fr{g})$ in fact is a \emph{braided} Drinfeld double (which is referred to as a \emph{double bosonization} there). It was proved in \cite{BW} that also two-parameter quantum groups are Drinfeld doubles.

In this paper, we aim to provide an axiomatic approach to the definition of (multiparameter) quantum groups by combining the pointed Hopf algebra and the triangular decomposition approach. Under the additional assumption of what we call a triangular decomposition of \emph{weakly separable type} over a group, the only indecomposable examples are close generalizations of multiparameter quantum groups. In particular, assuming further non-degeneracy, they are examples of a more general version of braided Drinfeld doubles, which we refer to as \emph{asymmetric} braided Drinfeld doubles. Further, under certain assumptions on the group and the parameters, we can recover Lie algebras from these Hopf algebras, after introducing a suitable integral form.

\subsection{This Paper's Results}

This paper starts by recalling the necessary technical background, including a brief overview on classification results of finite-dimensional pointed Hopf algebras, as well as structural results by \cite{BB} on algebras with triangular decomposition, in Section~\ref{background}. Next, we give the definition of a bialgebra with a triangular decomposition over a Hopf algebra $H$ in Section~\ref{section1.5}. This adapts the two-step approach used for algebras in \cite{BB} to the study of bialgebras. Namely, we first consider the \emph{free} case of a bialgebra $T(V)\otimes H\otimes T(V^*)$ where the positive and negative parts ($T(V)$, respectively $T(V^*)$) are tensor algebras, and then specify by what ideals (called \emph{triangular} Hopf ideals) we can take the quotient.

The core of this paper is formed by a partial classification of bialgebras with triangular decomposition over a group algebra $kG$. We assume that $V$ has one-dimensional homogeneous components (weak separability). We again proceed in two steps. First, we determine all pointed bialgebras with free positive and negative part over $kG$ in Section \ref{freeclassification}, and then look at pairs of ideals $I$, $I^*$ such that the quotient $A/{( I, I^*)}$ is still a bialgebra in Section \ref{quotientsection}. We find that indecomposable examples are automatically pointed Hopf algebras, and impose strong commutativity conditions on the group $G$. Multiparameter quantum groups fit into this framework. Indeed, the only possible commutator relations (\ref{commrel}) closely resemble those of multiparameter quantum groups:
\begin{align}
[f_i,v_j]&=\gamma_{ij}(k_j-l_i)\in kG, &\forall i=1,\ldots,n.
\end{align}

We further observe that there exists a natural generalization of the definition of a braided Drinfeld double to the setting of braided Hopf algebras in the category of Yetter--Drinfeld modules (YD-modules) over $H$. For this, the base Hopf algebra $H$ does not need to be quasitriangular. We need two braided Hopf algebras which are only required to be dually paired considered as braided Hopf algebra in the category of modules (rather than YD-modules). That is, the requirement that is weakened compared to the definition of a braided Drinfeld double (as in \cite{Maj2} or \cite{Lau}) is that the comodule structures do not need to be dually paired. We refer to this generalization as the \emph{asymmetric braided Drinfeld double}. It gives a natural way of producing Hopf algebras with triangular decomposition --- which are not necessarily quasitriangular. We show in Theorem \ref{drinfeldtheorem} that the Hopf algebras arising in the classification in Theorem~\ref{mainclassificationthm} are of this form (provided that the parameters $\gamma_{ii}$ are non-zero) and that $G$ has to be abelian in this case.

In  Section \ref{liealgebrasection} we show that from these asymmetric braided Drinfeld doubles of separable type we can recover Lie algebras provided that there exists a well-defined morphism of rings to $\mZ$ when setting the parameters equal to 1. Hence, in the spirit of the question asked in Section \ref{motivation}, we can relate the outcome of our classification back to Lie algebras, which are always generated by Lie subalgebras isomorphic to $\mathfrak{sl}_2$.

Here is an overview of the increasingly stronger assumptions on the Hopf algebras $A$ and $H$ used in the classification:
\begin{itemize}
\item Section~\ref{section1.5}: $H$ any Hopf algebra over a field $k$, $A$ a bialgebra with triangular decomposition;
\item Section~\ref{section2}: $H=kG$, $A$ a bialgebra with triangular decomposition;
\begin{itemize}
\item Section~\ref{preliminaryobs}--\ref{freeclassification}: $A$ is of weakly separable type and indecomposable after Definition~\ref{indecprop};
\item Section~\ref{quotientsection}: $A$ is indecomposable and non-degenerate of separable type;
\item Section~\ref{liealgebrasection}: In addition to the assumptions of \ref{quotientsection}, we require that $\operatorname{char} k=0$, and that setting the parameters equal to 1 gives a well-defined homomorphism of rings to $\mZ$.
\end{itemize}
\end{itemize}

The final section~\ref{multiparametersection} contains different classes of indecomposable pointed Hopf algebras with triangular decomposition over a group $kG$ that arise as examples in the main classification. The first class we discuss are the multiparameter quantum groups $U_{\lambda,\underline{p}}(\mathfrak{gl}_n)$ introduced by \cite{FRT} (adapting the presentation in \cite{CM}). They are asymmetric braided Drinfeld doubles, which is a generalization of the result of \cite{BW} showing that two-parameter quantum groups are Drinfeld doubles. In Section \ref{section3} we bring results of \cite{Ros} on growth condition (finite Gelfand--Kirillov dimension) and classification of Nichols algebras from \cite{AS3} into the picture. We  use these results to characterize the Drinfeld--Jimbo type quantum groups at generic parameters $q$ within the classification of this paper under the additional assumption that the triangular decomposition is what we call \emph{symmetric}. Further, a class of finite-dimensional pointed Hopf algebras by Radford can naturally be included as examples in this framework (Section \ref{radford}).

To conclude this paper, we suggest in Section \ref{conclusion} that future research could focus on the search for Hopf algebras with triangular decomposition over other Hopf algebras $H$ (replacing the group algebra $kG$). This might give interesting monoidal categories, or even knot invariants, in other contexts. As the first --- most classical --- example, we take $H$ to be a polynomial ring $k[x_1,\ldots, x_n]$. In this case, the only examples are universal enveloping algebras of Lie algebras.

\subsection{Notations and Conventions} In this paper, adapted Sweedler's notation (see e.g. \cite[1.2]{Swe}) is used to denote coproducts and coactions omitting summation. Unless otherwise stated, we work with Hopf algebras over an arbitrary field $k$. A Hopf algebra always has an invertible antipode $S$. The category of left YD-modules (or \emph{crossed modules}, cf. \cite[Proposition 7.1.6]{Maj1}) over a Hopf algebra $H$ is denoted by $\leftexpsub{H}{H}{\mathcal{YD}}$, while left modules are denoted by $\lmod{H}$, and right modules by $\rmod{H}$.

We denote the module spanned by generators $S$ over a commutative ring $R$ by $R\langle S \rangle$, while $R[S]$ denotes the $R$-algebra generated by elements $S$ (subject to some specified relations). Groups generated by elements of a set $S$ are denoted by $\langle S\rangle$, while ideals are denoted using $(~)$.

\section{Background}\label{background}

\subsection{Pointed Hopf Algebras}

Let the coproduct $\Delta\colon H\to H\otimes H$ make $H$ a coalgebra over a field $k$. We can consider \emph{simple} subcoalgebras $A\leq H$. That is, $\Delta(A)\leq A\otimes A$ and there are no proper subobjects of this type in $A$. A basic observation is that if $\dim A=1$, then $A$ can be written as $kg$, for a generator $g\in H$ such that $\Delta(g)=g\otimes g$. Such elements are called \emph{grouplike}. Indeed, if $H$ is a Hopf algebra, then the set of all grouplike elements $G(H)$ has a group structure. A Hopf algebra is \emph{pointed} if all simple subcoalgebras are one-dimensional. This notion can be traced back to \cite[8.0]{Swe} and classifying all finite-dimensional pointed Hopf algebras can be taken as a first step in the classification of all finite-dimensional Hopf algebras (see e.g. \cite{And} for a recent survey).

In the late 1980s and early 1990s, important classes of pointed Hopf algebras have been discovered with the introduction of the quantum groups (and their small analogues). Due to the vast applications of and attention to these Hopf algebras in the literature, the study of pointed Hopf algebras has become an important algebraic question.

\subsection{Link-Indecomposability}\label{indecomposability}
In the early 1990s, Montgomery asked the question, which groups may occur as $G(H)$ where $H$ is an \emph{indecomposable} pointed Hopf algebra. In \cite{Mon2}, an appropriate notion of indecomposability is discussed in different ways. We will briefly recall the description in terms of \emph{link-indecomposability} which is equivalent to indecomposability as a coalgebra and indecomposability of the Ext-quiver of simple comodules.

Given a pointed Hopf algebra $H$, we define a graph $\Gamma_H$ with vertices being the simple subcoalgebras of $H$ (that is, the grouplike elements). There is an edge $h\to g$ if there exists a $(g,h)$-skew-primitive element  $v\in H$, i.e. $\Delta(v)=v\otimes g+h\otimes v$, which is not contained in $kG(H)$. We say that $H$ is \emph{indecomposable} if $\Gamma_H$ is connected. As an example, group algebras $kG$ are only indecomposable if $G=1$. The quantum group $\Uq$ is indecomposable if the coproducts are e.g. defined as $\Delta(E)=E\otimes 1 + K\otimes E$ and $\Delta(F)=F\otimes 1 + K^{-1}\otimes F$. There are other versions of the coproduct which are not indecomposable (see \cite{Mon2}).

\subsection{Classification Results for Pointed Hopf Algebras}\label{classificationsurvey}
It was understood early that some pointed Hopf algebras can be obtained as bosonizations $A=\cB(V)\rtimes kG$ of so-called \emph{Nichols} (or \emph{Nichols-Woronowicz}) algebras $\cB(V)$ associated to YD-modules over a group $G$ (see e.g. \cite[Section~2]{AS} for definitions). In this case, the coproducts are given by $\Delta(v)=v^{(0)}\otimes v^{(-1)}+ 1 \otimes v$ using Sweedler's notation. That is, if $v$ is a homogeneous element, then $\Delta(v)=v\otimes g+1\otimes v$ for the degree $g\in G(A)$ of $v$ and $A$ is indecomposable over the group generated by $g\in G$ with $V_g\neq 0$. Thus, the question of finding finite-dimensional pointed Hopf algebras is linked to finding finite-dimensional Nichols algebras.\footnote{However, a pointed Hopf algebra is not necessarily a bosonization of this form. Important tools available are the coradical filtration (see e.g. \cite[5.2]{Mon}) and the \emph{lifting method} of Andruskiewitsch and Schneider \cite[Section~5]{AS}.} Although both questions remain open in general, vast progress on classifying pointed Hopf algebras has been made in a series of papers by Andruskiewitsch and Schneider (see \cite{AS,AS2}) for abelian groups $G$, and more recently for symmetric and alternating groups \cite{AFGV}, or groups of Lie type \cite{ACG1,ACG2}. See \cite{And} for more detailed references.

Let us briefly recall the classification results of \cite{AS2} over an algebraically closed field $k$ of characteristic zero here in order to provide the basis for comparison to this paper's classification in Section~\ref{section2} later. To fix notation, let $\cD$ denote a \emph{finite Cartan datum}. That is, a finite abelian group $\Gamma$, a Cartan matrix $A=(a_{ij})$ of dimension $n\times n$ with a choice of generating group elements $g_i$, and characters $\chi_i$ for $i=1,\ldots,n$. Then define $q_{ij}:=\chi_j(g_i)$ and impose the conditions that
\begin{equation}\label{cartandatum}
q_{ij}q_{ji}=q_{ii}^{a_{ij}}, \quad q_{ii}\neq 1.
\end{equation}

We can associate to the Cartan matrix $A$ a root system $\Phi$ (with positive roots $\Phi^+$). The simple roots $\alpha_i$ of $\Phi$ are indexed by $i=1,\ldots ,n$. Denote by $\chi$ the set of connected components of the corresponding Dynkin diagram, and by $\Phi_J$ the root system restricted to the component $J\in \chi$, and write $i\sim j$ if $i$ and $j$ are in the same connected component.
Denote further
\[
g_\alpha:= \prod_{i=1}^n{g_i^{n_i}}, \qquad \chi_\alpha:= \prod_{i=1}^n\chi_{i}^{n_i}, \qquad \text{for a root } \alpha=\sum_{i=1}^n{n_i\alpha_i}.
\]

To state the classification of finite-dimensional pointed Hopf algebras over abelian groups, some technical assumptions need to be made:
\begin{enumerate}
\item[(a)] Assume that the parameters $q_{ii}$ are roots of \emph{odd} order $N_i$.
\item[(b)] $N_i=N_J$ is constant on each connected component, $i\in J$.
\item[(c)] If $J\in \chi$ is of type $G_2$, then 3 does not divide $N_J$.
\end{enumerate}

To construct pointed Hopf algebra from a Cartan datum $\cD$, we need two families of parameters:
\begin{enumerate}
\item[(d)] Let $\lambda=(\lambda_{ij})$ be an $n\times n$-matrix of elements in $k$ such that for all $i \nsim j$, $g_ig_j=1$ or $\chi_i\chi_j\neq \varepsilon$ implies $\lambda_{ij}=0$.
\item[(e)] Further let $\mu=(\mu_\alpha)_{\Phi^+}$ be elements in $k$ such that for any $\alpha\in \Phi^+_J$, for $J\in \chi$, such that if $g_{\alpha}^{N_J}=1$ or $\xi_{\alpha}^{N_J}\neq \varepsilon$, then $\mu_{\alpha}=0$.
\end{enumerate}

\begin{definition}[{\cite[5.4]{AS2}}]\label{asform}
Given a Cartan datum $\cD$ with families of parameters $\lambda, \mu$ as above, there is a Hopf algebra $u=u(\cD,\lambda,\mu)$. The algebra $u$ is generated by elements $g\in \Gamma$ (to define $u_\alpha(\mu)\in k\Gamma$ for $\alpha\in \Phi^+$, see \cite[2.14]{AS2}), and $x_i$ for $i=1,\ldots, n$, subject to the relations
\begin{align}
gx_i&=\chi_i(g)x_ig,\qquad &\text{for all $i$, $g\in \Gamma$},\label{asrel1}\\
\underline{\ad}(x_i)^{1-a_{ij}}&=0, \qquad &\text{for $i\neq j$, $i\sim j$},\label{asrel2}\\
\underline{\ad}(x_i)(x_j)&=\lambda_{ij}(1-g_ig_j), \qquad &\text{for all $i<j$, $i\nsim j$},\label{asrel3}\\
x_\alpha^{N_J}&=u_\alpha(\mu), \qquad &\text{for all $\alpha\in \Phi_J^+$, $J\in \chi$}.\label{asrel4}
\end{align}
Here, $\underline{\ad}(x)(y)$ is the \emph{braided} commutator $xy-m\circ \Psi(x\otimes y)$ where $m$ denotes multiplication and $\Psi$ is the YD-braiding.
The comultiplication is given by $\Delta(x_i)=x_i\otimes 1 + g_i\otimes x_i.$
\end{definition}

\begin{theorem}[{\cite[0.1]{AS}}]
Under the above assumptions (a)--(e) on a Cartan datum $\cD$ with parameters $\lambda$, $\mu$, the Hopf algebra $u(\cD,\lambda, \mu)$ is indecomposable and pointed with $G(u)=\Gamma$ and has finite dimension. 

Moreover, if $\abs{G}$ is not divisible by $2,3,5$ or $7$, then any indecomposable finite-dimensional pointed Hopf algebra over $kG$, where $G$ is abelian, and $k=\overline{k}$, $\operatorname{char} k=0$, is of this form.
\end{theorem}

\subsection{Algebras with Triangular Decomposition (Free Case)}

A triangular decomposition of algebras means that an intrinsic PBW decomposition exists, similar to universal enveloping algebras of Lie algebras. This is a common feature of quantum groups and rational Cherednik algebras, but more generally shared by all braided Drinfeld or Heisenberg doubles (cf. \cite[Section~3]{Lau}). Here, we are using the definitions introduced in \cite{BB} to study such algebras with triangular decomposition (so-called \emph{braided doubles}).

From a deformation-theoretic point of view, triangular decomposition can be viewed as follows. Let $V$, $V^*$ be dually paired finite-dimensional vector spaces and $H$ a Hopf algebra over a field $k$, such that $V$ is a left $H$-module, and $V^*$ carries the dual right $H$-action. That is, for the evaluation map $\langle ~,~\rangle\colon V^*\otimes V\to k$, we have
\begin{equation}
\langle f\triangleleft h,v\rangle=\langle f, h\triangleright v\rangle, \qquad \forall f\in V^*, v\in V, h\in H.
\end{equation}
Now define $A_0(V,V^*)$ to be the algebra on $T(V)\otimes H \otimes T(V^*)$ with relations 
\begin{equation}\label{boson}
fh=h_{(1)}(f\triangleleft h_{(2)}), \qquad hv=(h_{(1)}\triangleright v) h_{(2)},
\end{equation}
(i.e. the bosonizations $T(V)\rtimes H$ and $H\ltimes T(V^*)$ are subalgebras), and $[f,v]=0$.

In \cite[3.1]{BB}, a family of deformations of $A_0(V,V^*)$ over $\Hom_k(V^*\otimes V,H)$ is defined. The algebra $A_\beta(V,V^*)$, for a parameter $\beta\colon V^*\otimes V\to H$, is defined using the same generators in $V$, $V^*$ and $H$ with the same bosonization relations, but the commutator relations
\begin{equation}
[f,v]=\beta(f,v).
\end{equation}
In order to obtain flat deformations we restrict to maps $\beta$ such that multiplication
\begin{align*}
m\colon T(V)\otimes H\otimes T(V^*)&\stackrel{\sim}{\longrightarrow} A_{\beta}(V,V^*),&v\otimes h\otimes f&\mapsto vhf,
\end{align*}
gives an isomorphism of $k$-vector spaces.

\begin{definition}
In the case where $m$ gives such an isomorphism of $k$-vector spaces, we say that $A_\beta(V,V^*)$ is a \emph{free braided double}.
\end{definition}

\begin{theorem}[{\cite[Theorem 3.3]{BB}}]\label{bbthm}
The algebra $A_\beta(V,V^*)$ is a free braided double if and only if there exists a $k$-linear map $\delta\colon V\to H\otimes V$, $\delta(v)=v^{[-1]}\otimes v^{[0]}$ which is YD-compatible with the $H$-action on V, i.e. for any $h\in H$
\begin{equation}\label{ydcond}
h_{(1)}v^{[-1]}\otimes (h_{(2)}\triangleright v^{[0]})=(h_{(1)}\triangleright v)^{[-1]}h_{(2)}\otimes (h_{(1)}\triangleright v)^{[0]}.
\end{equation}
In this case, we call $(V,\delta)$ a \emph{quasi-YD-module} and we have
\begin{equation}\label{commrel}
[f,v]=\beta(f\otimes v)=v^{[-1]}\langle f,v^{[0]}\rangle.
\end{equation}
\end{theorem}

Note that $A_{\beta}(V,V^*)$ is a graded algebra where $\deg v=1$, $\deg h=0$, and $\deg f=-1$, for all $v\in V$, $h\in H$, and $f\in V^*$.

\subsection{Triangular Ideals}\label{triangularideals} So far, the braided Hopf algebras $T(V)$ and $T(V^*)$ were assumed to be free. We can bring additional relations into the picture, defining \emph{braided doubles} that are not necessarily free. Let $I\triangleleft T(V)$ and $I^*\triangleleft T(V^*)$ be ideals. We want to determine when the quotient map
\[
m\colon T(V)/I\otimes H \otimes T(V^*)/{I^*}\stackrel{\sim}{\longrightarrow} A_\beta(V, V^*)/{( I, I^*)}
\]
is still a graded isomorphism of $k$-vector spaces. In \cite[Appendix~A]{BB} it is show that this is the case if and only if $J:=( I, I^*)$ is a so-called \emph{triangular ideal}.
That is, $J=I\otimes H \otimes T(V^*)+T(V)\otimes H \otimes I^*$, where $I\triangleleft T^{>1}(V)$, $I^*\triangleleft T^{>1}(V^*)$ are homogeneously generated ideals such that $I$ and $I^*$ are $H$-invariant and
\begin{equation}\label{idealcond}
T(V^*)I\leq J, \qquad I^*T(V)\leq J.
\end{equation} 
This condition is equivalent to the commutators $[f,I]$ and $[I^*,v]$ being contained in $J$ for all elements $v\in V$, $f\in V^*$.
For each quasi-YD-module, there exists a unique largest triangular ideal $I_{\op{max}}$, and thus a unique maximal quotient referred to as the \emph{minimal braided double} of $V$.

If $\delta$ is a YD-module, then the maximal quotient $T(V)/{I_{\op{max}}}$ is the Nichols algebra $\cB(V)$ of $V$, and the braided double on $\cB(V)\otimes H \otimes\cB(V^*)$ is a generalization of the Heisenberg double, a so-called \emph{braided Heisenberg double}.

For the purpose of this paper, we need ideals $I$ such that $T(V)/I$ is a braided bialgebra, where $V$ is a YD-module. That is, not a bialgebra object in the category of $k$-vector spaces but in the category of YD-modules over $kG$ (see e.g. \cite[1.2--1.3]{AS}). In fact, if $I$ is a homogeneously generated ideal in $T^{>1}(V)$ which is a coideal and a YD-submodule, then $T(V)/I$ is a braided Hopf algebra. We denote the collection of such ideals by $\cI_V$. In particular $I_{\op{max}}\in \cI_V$ as the Nichols algebra $\cB(V)$ is a braided Hopf algebra.


\section{Hopf Algebras with Triangular Decomposition}\label{section1.5}

In this section, we let $k$ be a field of arbitrary characteristic and $H$ a Hopf algebra over $k$. We introduce a notion of a Hopf algebra with triangular decomposition over $H$.

\subsection{Definitions}\label{definitions}

We refer to the  grading of a braided double $T(V)/I\otimes H\otimes T(V^*)/{I^*}$ given by
\[
\deg v=1,\quad \deg f=-1, \quad \deg h=0, \qquad \forall v\in V, ~f\in V^*, ~h\in H,
\]
as the \emph{natural grading}. We want to study Hopf algebras with triangular decomposition preserving this grading.

\begin{definition}\label{triangulardecompdefn}
A bialgebra (or Hopf algebra) $A$ with \emph{triangular decomposition} over a Hopf algebra $H$ is a braided double $A=T(V)/I\otimes H\otimes T(V^*)/{I^*}$ which is a bialgebra (respectively Hopf algebra) such that

\begin{align}
\bullet~&\text{$H$ is a subcoalgebra of $A$ with respect to the original coproduct of $H$},\label{assum0}\\
\begin{split}\bullet~&\text{the subspaces $T(V)\otimes H$ and $H\otimes T(V^*)$ are closed under the}\\&\text{coproduct of $A$,}
\end{split}\label{assum1}\\
\begin{split}\bullet~&\text{the coproduct and counit are morphisms of graded algebras}\\ &\text{for the natural grading.}
\end{split}\label{assum2}
\end{align}
(In the Hopf case, the antipode $S$ is required to preserve the natural grading and the subspaces $T(V)\otimes H$ and $H\otimes T(V^*)$.)
\end{definition}
Note that (\ref{assum2}) implies that $\varepsilon(v)=\varepsilon(f)=0$ for all $v\in V$, $f\in V^*$. We further observe that assumptions (\ref{assum1}) and (\ref{assum2}) combined with the counit property, give that $\Delta(V)\leq H\otimes V+V\otimes H$ as well as $\Delta(V^*)\leq H\otimes V^*+V^*\otimes H$. Consider the compositions $\delta_r, \delta_l$ with projections in
\[
\xymatrix{
&\ar[dl]_{\delta_l}\ar[d]^{\Delta} V \ar[rd]^{\delta_r}&\\
H\otimes V&\ar@{->>}[l]^-{p_1}H\otimes V\oplus V\otimes H\ar@{->>}[r]_-{p_2}&V\otimes H
.}
\]
The coalgebra axioms imply that $\delta_l$ and $\delta_r$ are left (respectively right) $H$-coactions. In particular, as the semidirect product relations in $A$ are preserved by $\Delta$, $\delta_l$ (and $\delta_r$) are left (respectively right) YD-compatible with the given actions of $H$ on $V$ (right action via antipode).
Similarly, we can obtain a left and right YD-module structure over $H$ on the dual $V^*$ from the coproduct. The corresponding coactions are denoted by $\delta_l^*$ and $\delta_r^*$.

\begin{definition}
Given a bialgebra $A$ with triangular decomposition over $H$, we define the \emph{right (respectively, left) YD-structure} of $A$ to be $\delta_r$ (respectively, $\delta_l$) together with the given $H$-actions. We refer to $\delta_r^*$ and $\delta_l^*$ (with the dual $H$-actions) as the right and left \emph{dual} YD-structures.
\end{definition}

To fix Sweedler's notation for the different coactions, denote $\delta_r(v)=v^{(0)}\otimes v^{(-1)}$ and $\delta_l(v)=v^{\overline{(-1)}}\otimes v^{\overline{(0)}}$ and use similar notations for $f\in V^*$. We will reformulate the definition of a bialgebra with triangular decomposition in terms of conditions on the YD-structures of $A$ in (\ref{eqn1})--(\ref{eqn5}) in the free case first.

\begin{lemma}\label{hopflemma}
A bialgebra $A$ with triangular decomposition over $H$ is a Hopf algebra with triangular decomposition if and only if there exists a $k$-linear map $S\colon V\oplus V^*\to V\otimes H\oplus V^*\otimes H$ such that
\begin{align}
\begin{split}
S(v^{(0)})v^{(-1)}+(S(v^{\overline{(-1)}})_{(1)}\triangleright v^{\overline{(0)}})S(v^{\overline{(-1)}})_{(2)}&=0, \\
v^{(0)}S(v^{(-1)})+({v^{\overline{(-1)}}}_{(1)}\triangleright S(v^{\overline{(0)}})){v^{\overline{(-1)}}}_{(2)}&=0,
\end{split}&\forall v\in V,\label{antipodecond1}\\
\begin{split}
f^{\overline{(-1)}}S(f^{\overline{(0)}})+S(f^{(-1)})_{(1)}(f^{(0)}\triangleleft S(f^{(-1)})_{(2)})&=0,\\
S(f^{\overline{(-1)}})f^{\overline{(0)}}+{f^{(-1)}}_{(1)}(S(f^{(0)})\triangleleft {f^{(-1)}}_{(2)})&=0,
\end{split}&\forall f\in V^*.\label{antipodecond2}
\end{align}
In this case, $S$ extends uniquely to an antipode on all of $A$.
\end{lemma}
\begin{proof}
This follows (under use of the semidirect product relations) by restating the antipode axioms for the coproduct of a Hopf algebra with triangular decomposition, in which the coproducts have the form $\Delta(v)=v^{(0)}\otimes v^{(-1)}+v^{\overline{(-1)}}\otimes v^{\overline{(0)}}$. Note that $\varepsilon(v)=0$ as we require the counit to be a morphism of graded algebras.
\end{proof}

\subsection{The Free Case}\label{freesection}

Let $A$ be a \emph{free} braided double, i.e. $A=T(V)\otimes H\otimes T(V^*)$. We can now state necessary and sufficient conditions on the YD-structures of $A$ to make the algebra $A$ a bialgebra with triangular decomposition. In the following, we stick to the notation of \cite[Definition~2.1]{BB} denoting the quasi-coaction determining the commutator relations between elements of $V$ and $V^*$ by $\delta(v)=v^{[-1]}\otimes v^{[0]}$, for $v\in V$.

\begin{lemma}
A free braided double $A$ on $T(V)\otimes H\otimes T(V^*)$ is a bialgebra with triangular decomposition if and only if there exist YD-structures $\delta_l$, $\delta_r$, $\delta_l^*$, and $\delta_r^*$ such that the following conditions hold for $v\in V$, $f\in V^*$:
\begin{align}
(f^{(0)}\triangleleft {v^{\overline{(-1)}}})\otimes ({f^{(-1)}}\triangleright v^{\overline{(0)}})&=f\otimes v,\label{eqn1}\\
({f^{\overline{(-1)}}}\triangleright v^{(0)})\otimes (f^{\overline{(0)}}\triangleleft {v^{(-1)}})&=v\otimes f,\label{eqn2}\\
v^{(0)}f^{(0)}\otimes (f^{(-1)}v^{(-1)}-v^{(-1)}f^{(-1)})&=0,\label{eqn3}\\
(f^{\overline{(-1)}}v^{\overline{(-1)}}-v^{\overline{(-1)}}f^{\overline{(-1)}})\otimes v^{\overline{(0)}}f^{\overline{(0)}}&=0,\label{eqn4}\\
\begin{aligned}&v^{(0)[-1]}\langle f^{(0)}, v^{(0)[0]}\rangle\otimes f^{(-1)}v^{(-1)}\\&+ f^{\overline{(-1)}}v^{\overline{(-1)}}\otimes v^{\overline{(0)}[-1]}\langle f^{\overline{(0)}}, v^{\overline{(0)}[0]}\rangle\end{aligned}&= v^{[-1]}\otimes v^{[-1]}\langle f, v^{[0]} \rangle. \label{eqn5}
\end{align}
\end{lemma}
\begin{proof}
The conditions are easily checked --- under use of the relations in $A$ and the PBW theorem --- to be equivalent to the requirement that (\ref{commrel}) is preserved by $\Delta$. This gives the relations (\ref{eqn3})--(\ref{eqn5}), as well as 
\begin{align*}
{v^{\overline{(-1)}}}_{(1)}(f^{(0)}\triangleleft {v^{\overline{(-1)}}}_{(2)})\otimes ({f^{(-1)}}_{(1)}\triangleright v^{\overline{(0)}}){f^{(-1)}}_{(2)}&=v^{\overline{(-1)}}f^{(0)}\otimes v^{\overline{(0)}}f^{(-1)},\\
({f^{\overline{(-1)}}}_{(1)}\triangleright v^{(0)}){f^{\overline{(-1)}}}_{(2)}\otimes {v^{(-1)}}_{(1)}(f^{\overline{(0)}}\triangleleft {v^{(-1)}}_{(2)})&=v^{(0)}f^{\overline{(-1)}}\otimes v^{(-1)}f^{\overline{(0)}}.
\end{align*}
These relations are equivalent to (\ref{eqn1}) and (\ref{eqn2}) under use of the counit of $H$, applying the coaction axioms. 

Further, given $\delta_r$, $\delta_l$ as well as their dual counterparts $\delta_r^*$, $\delta_l^*$, the bosonization relations are preserved by the coproduct defined as 
\begin{align}
\Delta(v)&=v^{(0)}\otimes v^{(-1)}+v^{\overline{(-1)}}\otimes v^{\overline{(0)}}, &\Delta(f)&=f^{(0)}\otimes f^{(-1)}+f^{\overline{(-1)}}\otimes f^{\overline{(0)}},
\end{align}
for $v\in V$, $f\in V^*$ by YD-compatibility.
\end{proof}

It will become apparent in Section~\ref{section2} what constraints on the structure of $A$ conditions (\ref{eqn1})--(\ref{eqn5}) give when working over $H$ a group algebra, and over a polynomial ring in Section~\ref{conclusion}.

\subsection{Triangular Hopf Ideals}\label{triangularhopfsect}

We are looking for triangular ideals $J=I\otimes H\otimes T(V^*)+T(V)\otimes H\otimes I^*$ (cf. \cite[Appendix~A]{BB} or Section \ref{triangularideals}) which are also coideals, and hence $A/J$ is a bialgebra or Hopf algebra with a triangular decomposition. Using the description of the coproduct $\Delta$ in terms of the left and right YD-structures on $A$, the triangular ideals $J$ that are also coideals are simply those triangular ideals for which $I$ (and $I^*$) are YD-submodules for both $\delta_l$ and $\delta_r$ (respectively, $\delta_l^*$ and $\delta_r^*$). If $A$ is a triangular Hopf algebra with antipode given as in Lemma~\ref{hopflemma}, then every triangular ideal which is also a coideals is automatically a Hopf ideal.

\begin{definition}
We denote the collection of triangular ideals of the form
\[
J=I\otimes H\otimes T(V^*)+T(V)\otimes H\otimes I^*
\]
for homogeneously generated $I\triangleleft T^{>1}(V)$ and $I^*\triangleleft T^{>1}(V^*)$ which are also YD-submodules for $\delta_r$, $\delta_l$ (respectively for $\delta_r^*$, $\delta_l^*$)
by $\cI_\Delta(A)$. Such ideals $J$ are called \emph{triangular Hopf ideals}.
\end{definition}

\subsection{Asymmetric Braided Drinfeld Doubles}\label{asymdrin}

A special class of Hopf algebras with triangular decomposition can be provided by braided Drinfeld doubles of primitively generated Hopf algebras over a quasitriangular base Hopf algebra $H$. This form of the Drinfeld double was introduced as the \emph{double bosonization} in \cite{Maj1,Maj2}, see also \cite[Section 3.5]{Lau} for the presentation used here. We now give a more general definition of an \emph{asymmetric} braided Drinfeld double which is suitable to capture the more general class of Hopf algebras that we find in Section~\ref{section2}, including multiparameter quantum groups, as examples. In this construction, the base Hopf algebra $H$ need not be quasitriangular, and the asymmetric braided Drinfeld double is also not quasitriangular in general.

To define the braided Drinfeld double of dually paired braided Hopf algebras $C$ and $B$ in the category $\lmod{\Drin(H)}=\leftexpsub{H}{H}{\mathcal{YD}}$ we require that $\langle~,~\rangle \colon C\otimes B\to k$ is a morphism of YD-modules. This implies that the actions and coactions on $C$ and $B$ are dual to one-another (by means of the antipode of $H$). A weaker requirement is that we consider the images of $C$ and $B$ under the forgetful functor
\[
F\colon \leftexpsub{H}{H}{\mathcal{YD}}\longrightarrow\lmod{H},
\]
and require that $F(C)$ and $F(B)$ are dually paired Hopf algebras in $\lmod{H}$ (with the induced braiding under $F$), while $C$ and $B$ may not be dually paired in $\leftexpsub{H}{H}{\mathcal{YD}}$. Hence the coactions on $C$ and $B$ do not necessarily have to be related via the antipode, but the actions and resulting braidings need to be related by duality. This is captured by the following definition, where we denote the left coactions by $c\mapsto c^{(-1)}\otimes c^{(0)}$ and $b\mapsto b^{(-1)}\otimes b^{(0)}$ respectively. 

\begin{definition}
We say that two braided Hopf algebras $C,B$ in $\HYD$ are \emph{weakly dually paired} if there exists a morphism of $H$-modules $\langle~,~\rangle\colon C\otimes B\to k$ such that
\begin{align}
\langle cc',b\rangle&=\langle c',b_{(1)}\rangle\langle c,b_{(2)}\rangle ,& \langle c,bb'\rangle&=\langle c_{(1)},b'\rangle\langle c_{(2)},b\rangle,
\end{align}
for all $c,c'\in C$, and $b,b'\in B$; as well as 
\begin{align}\label{weaklyduallypairedcond}
(c^{(-1)}\triangleright b)c^{(0)}&=b^{(0)}(b^{(-1)}\triangleright c).
\end{align}
\end{definition}

This weaker duality is equivalent to an analogue of condition (\ref{eqn2}). To see this, we can regard the left $H$-coaction on $B$ as a right $H^{\cop}$-coaction, over the co-opposite Hopf algebra $H^{\cop}$ with coproduct $\tau\Delta$. Given a left $H$-action $\triangleright$, we define a right $H^{\cop}$-action $\triangleleft:=\triangleright(S^{-1}\otimes \ide)\tau$ (where $\tau$ denotes the $\otimes$-symmetry in $\Vect$). The resulting structures make $B$ a right YD-module over $H^{\cop}$. The analogue of condition (\ref{eqn2}) can be rephrased as requiring for all $b\in B, c\in C$ that
\begin{align}
&&b^{(0)}c^{(-1)}\otimes b^{(-1)}c^{(0)}&=c^{(-1)}b^{(0)}\otimes c^{(0)} b^{(-1)},\nonumber\\
&\Longleftrightarrow& b^{(0)}c^{(-1)}\otimes b^{(-1)}c^{(0)}&=({c^{(-1)}}_{(1)}\triangleright b^{(0)}){c^{(-1)}}_{(2)}\otimes {b^{(-1)}}_{(1)}(c^{(0)} \triangleleft {b^{(-1)}}_{(2)}), \label{compeqn2}\\
&\Longleftrightarrow& bc&=(c^{(-1)}\triangleright b^{(0)})(c^{(0)}\triangleleft b^{(-1)}),\label{compeqn}\\
&\Longleftrightarrow& (c^{(-1)}\triangleright b)c^{(0)}&=b^{(0)}(c\triangleleft S(b^{(-1)}))=b^{(0)}(b^{(-1)}\triangleright c),\nonumber
\end{align}
which gives condition (\ref{weaklyduallypairedcond}). We can visualize conditions (\ref{compeqn2}) and (\ref{compeqn}) using graphical calculus (with the conventions from \cite{Lau}), see Fig. \ref{braidingcomp}.
\begin{figure}[h]
\[
\begin{array}{ccc}
\vcenter{\hbox{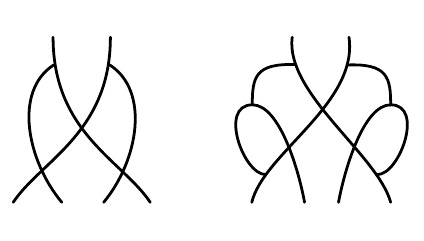}}&\Longleftrightarrow& \vcenter{\hbox{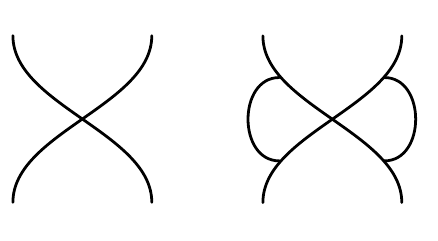}}.
\end{array}
\]
\caption{Left and right braiding compatibility condition}
\label{braidingcomp}
\end{figure}

Given (\ref{weaklyduallypairedcond}), we can define an analogue of the braided Drinfeld double on the $k$-vector space $B\otimes H\otimes C$ (rather than using $B\otimes \Drin(H)\otimes C$) with this weaker requirement of duality on $C$ and $B$. The definition of the \emph{asymmetric} braided Drinfeld double can be given using Tannakian reconstruction theory by describing their category of modules. This is similar to the approach used for the braided Drinfeld double in \cite[Appendix~B]{Maj2} (cf. also \cite[Section 3.2]{Lau}). 

\begin{definition}
Let $C,B$ be weakly dually paired braided Hopf algebras in $\HYD$. We define the category $\YDasy{C}{B}{H}$ of \emph{asymmetric YD-modules} over $C,B$ as having objects $V$ which are left $H$-modules (also viewed as right modules by means of the inverse antipode), equipped with a left $C$-action and a right $B$-action (by morphisms of $H$-modules) which satisfy the compatibility condition
\begin{align}\label{assydcond}
((c_{(2)}\triangleright v)\triangleleft {b_{(1)}}^{(-1)})\triangleleft b_{(2)}\langle c_{(1)}, {b_{(1)}}^{(0)}\rangle &=c_{(1)}\triangleright ({c_{(2)}}^{(-1)}\triangleright(v\triangleleft b_{(1)}))\langle {c_{(2)}}^{(0)},b_{(2)}\rangle,
\end{align}
for all $v\in V, b\in B, c\in C$. Morphisms in $\YDasy{C}{B}{H}$ are required to commute with the actions of $H$, $B$ and $C$. 
\end{definition}

It may help to visualize the condition (\ref{assydcond}) using graphical notation, see Fig. \ref{asymydpicture}.
\begin{figure}[h]
\begin{center}
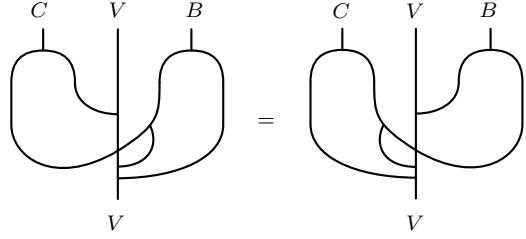~.
\end{center}
\caption{Asymmetric Yetter--Drinfeld modules}
\label{asymydpicture}
\end{figure}

\begin{proposition}
The category $\YDasy{C}{B}{H}$ is monoidal, with a commutative diagram of monoidal fiber functors
\[
\xymatrix{&\rmod{B}(\lmod{H})\ar[dr]&&\\
\YDasy{C}{B}{H}\ar[ur]\ar[dr] &&\lmod{H}\ar[r]&\Vect.\\
&\lmod{C}(\lmod{H})\ar[ur]&&
}
\]
\end{proposition}
\begin{proof}
This monadicity statement can for example be checked directly using graphical calculus. Note that condition (\ref{compeqn}) is crucial. The fiber functors simply forget the additional structure at each step.
\end{proof}

\begin{definition}\label{asymmetricdrinfelddef}
The \emph{asymmetric braided Drinfeld double} $\aDrin_H(C,B)$ is defined as the algebra obtained by Tannakian reconstruction\footnote{See e.g. \cite[9.4.1]{Maj1} or \cite[Section 2.3]{Lau}.}  on $B\otimes H\otimes C$ applied to the functor $\YDasy{C}{B}{H}\longrightarrow \Vect$.
Hence $\lmod{\aDrin_H(C,B)}$ and $\YDasy{C}{B}{H}$ are canonically equivalent as categories.
\end{definition}

\begin{proposition}\label{asymmetricdrinrel}
An explicit presentation for the asymmetric braided Drinfeld double $\aDrin_H(C,B)$ on the $k$-vector space $B\otimes H\otimes C$ can be given as follows: the multiplication on $B$ is opposite, and for $c\in C$, $b\in B$ and $h\in H$ we have
\begin{align}
hb&=(h_{(2)}\triangleright b)h_{(1)},\\hc&=(h_{(1)}\triangleright c)h_{(2)},\\
b_{(2)}S^{-1}({b_{(1)}}^{(-1)})c_{(2)}\langle c_{(1)}\otimes {b_{(1)}}^{(0)}\rangle &=c_{(1)}{c_{(2)}}^{{(-1)}}b_{(1)}\langle {c_{(2)}}^{{(0)}}\otimes b_{(2)}\rangle.
\end{align}
The coproducts are given by
\begin{align}
\Delta(h)&=h_{(1)}\otimes h_{(2)},\\
\Delta(b)&={b_{(1)}}^{(0)}\otimes b_{(2)}S^{-1}({b_{(1)}}^{(-1)}),\\
\Delta(c)&={c_{(1)}}{c_{(2)}}^{{(-1)}}\otimes {c_{(2)}}^{{(0)}},
\end{align}
and the antipode is
\begin{align}
S(h)&=S(h), &S(b)&=S^{-1}(b^{(0)})b^{(-1)}, &S(c)&=S(c^{{(-1)}})S(c^{{(0)}}),
\end{align}
using the respective given structures on $H$, $B$, and $C$.
\end{proposition}
\begin{proof}
This follows under application of reconstruction (in $\Vect$) applied to $\YDasy{C}{B}{H}$. See e.g \cite[Section 2.3]{Lau} for formulas on how to obtain the structures, including the antipode (Figure 2.1).
\end{proof}

An important feature of the braided Drinfeld double is that it has a braided monoidal category of representations, hence is quasitriangular. For the \emph{asymmetric} braided Drinfeld double to be quasitriangular, we need $H$ to be quasitriangular. If $H$ is not quasitriangular, this can be achieved by working with over $\Drin(H)$ instead of $H$ as a base Hopf algebra.


From now on, we restrict to the important special case where $B$ and $C$ are primitively generated by finite-dimensional YD-modules. This way, we obtain examples of Hopf algebras with a triangular decomposition over $H$.

\begin{lemma}\label{asymmetricdouble}
Let $V$, $V^*$ be left YD-modules over $H$, such that the action on $V^*$ is dual to the action on $V$. Then the braided tensor (co)algebras $T(V)^{\oop}$ and $T(V^*)^{\cop}$ are dually paired\footnote{We choose the opposite $T(V)^{\oop}$ and co-opposite $T(V^*)^{\cop}$ in order to avoid having to take the opposite multiplication in the resulting double (cf. Proposition \ref{asymmetricdrinrel}). As tensor algebras are braided cocommutative, this choice does not affect the formulas for the coproduct.} in the monoidal category of right modules over $H$. Further assume that the compatibility condition (\ref{weaklyduallypairedcond}) holds.

Then the asymmetric braided Drinfeld double $\aDrin_H(T(V^*)^{\cop},T(V)^{\op{op}})$ is given on $A=T(V)\otimes H\otimes T(V^*)$ subject to the usual bosonization relations (\ref{boson}) and the cross relation
\begin{align}\label{asymmetriccomm}
[f,c]=S^{-1}(v^{(-1)})\langle f,v^{(0)}\rangle-f^{(-1)}\langle f^{{(0)}},v\rangle.
\end{align}
The coalgebra structure is given by
\begin{align}
\Delta(v)&=v^{(0)}\otimes S^{-1}(v^{(-1)})+1\otimes v,&\Delta(f)&=f\otimes 1+f^{{(-1)}}\otimes f^{{(0)}}.
\end{align}
The counit is given by $\varepsilon(v)=\varepsilon(f)=0$ and the antipode can be computed using the conditions from equations (\ref{antipodecond1}) and (\ref{antipodecond2}) as
\begin{align}
S(v)&=-v^{(0)}v^{(-1)},&S(f)&=-S(f^{{(-1)}})f^{{(0)}}.
\end{align}
We can also consider quotients of the form $A/J$ for any triangular Hopf ideal $J\in \cI_\Delta(A)$. The quotient of $A$ by the maximal triangular Hopf ideal in $\cI_\Delta(A)$ is denoted by $U_H(V,V^*)$.
\end{lemma}

\begin{lemma}\label{ideallemma}
Let $A=\aDrin_H(T(V^*)^{\cop}, T(V)^{\oop})$ for $V$, $V^*$ as in Lemma \ref{asymmetricdouble}. Then the maximal ideal $I_{\op{max}}(A)$ in $\cI_{\Delta}(A)$ is given by
\[
I_{\max}(A)=I_{\op{max}}(V)\otimes H\otimes T(V^*)+T(V)\otimes H\otimes I_{\op{max}}(V^*),
\]
where $I_{\op{max}}(V)$ is the maximal Nichols ideal in $T(V)$ for the left coaction on $V$, and $I_{\op{max}}(V^*)$ is the maximal Nichols ideal in $T(V^*)$ for the left coaction on $V^*$. Hence
\[
m\colon \cB(V)\otimes H\otimes \cB(V^*)\stackrel{\sim}{\longrightarrow}U_H(V,V^*)
\]
is an isomorphism of $k$-vector spaces (PBW theorem).
\end{lemma}
\begin{proof}
This is clear as we know that $T(V)^{\oop}/{I_{\op{max}}(V)}$ and $T(V^*)^{\cop}/{I_{\op{max}}(V^*)}$ are weakly dually paired braided Hopf algebras and their asymmetric braided Drinfeld double is given by the quotient $\aDrin_H(T(V^*)^{\cop},T(V)^{\oop})/{I_{\max}(A)}$, which must be the minimal double $U_H(V,V^*)$.
\end{proof}

A perfect pairing between the positive and negative part of $U_H(V,V^*)$
implies the existence of a formal power series $\coev$ satisfying the axioms of coevaluation. We expect that this can be used to give a braiding on a suitable category of modules over $U_H(V, V^*)$ (where $\cB(V)$ acts integrally), and all modules have the structure of being YD-modules over $H$. 


\subsection{Symmetric Triangular Decompositions}

The rest of this section will be devoted to the question of recovering the braided Drinfeld double over a quasitriangular base Hopf algebra $H$ as a special case of the asymmetric braided Drinfeld double. For this, we introduce the idea of a Hopf algebra with a \emph{symmetric} triangular decomposition:

\begin{definition}
Let $A$ be a bialgebra with a triangular decomposition over $H$. If the associated coactions satisfy that the right coaction $\delta_r^*$ of $V^*$ is the dual coaction to $\delta_l$, i.e.
\begin{equation}\label{symmetry}
\langle f^{(0)}\otimes v\rangle f^{(-1)}=\langle f\otimes v^{\ov{(0)}}\rangle v^{\ov{(-1)}},
\end{equation}
and the coactions $\delta_r$ and $\delta_l^*$ are compatible in the same way, then we call the triangular decomposition \emph{symmetric}.
\end{definition}

In the case where $H$ is a quasitriangular Hopf algebra, we can recover a special case of the definition of the braided Drinfeld double given in \cite[Example 3.5.6]{Lau} from the more general form given in Definition \ref{asymmetricdrinfelddef}, and the resulting triangular decomposition will be symmetric. For this, note that the universal $R$-matrix and its inverse give functors (see \cite{Maj13})
\begin{align*}
R^{-1}\colon \lmod{H}&\longrightarrow \leftexpsub{H}{H}{\mathcal{YD}}, &(V,\triangleright) &\longmapsto (V,\triangleright, (\ide_H\otimes \triangleright)(R^{-1}\otimes \ide_V)),\\
R\colon \rmod{H}&\longrightarrow \leftexpsub{H}{H}{\mathcal{YD}}, &(V,\triangleleft) &\longmapsto (V,\triangleleft, ( \triangleleft\otimes\ide_H)(\ide_V\otimes R)).
\end{align*}
Given a right $H$-module $V$, we can hence give $V$ a left YD-module structure using $R^{-1}$, and $V^*$ the dual YD-module structure using $R$. Note that (\ref{weaklyduallypairedcond}) is satisfied in this case. With these structures, the relation (\ref{asymmetriccomm}) becomes
\begin{align*}
[f,c]&=S^{-1}(R^{-(2)})\langle f,v \triangleleft R^{-(1)}\rangle -R^{-(1)}\langle R^{-(2)}\triangleright f, v\rangle \\
&=R^{(2)}\langle R^{(1)}\triangleright f,v\rangle -R^{-(1)}\langle R^{-(2)}\triangleright f,v\rangle.
\end{align*}
This is precisely the cross relation of \cite[Example 3.5.6]{Lau}. Note that we use $R=(S^{-1}\otimes \ide_H)R^{-1}$. This proves the following Proposition:

\begin{proposition}\label{recoverdouble}
Braided Drinfeld doubles of braided Hopf algebras over a qua\-si\-triangular Hopf algebras are asymmetric braided Drinfeld doubles (as in Definition \ref{asymmetricdrinfelddef}) with a symmetric triangular decomposition. 
\end{proposition}

Note that a partial converse statement also holds: Given an asymmetric braided Drinfeld double that is symmetric, then it can be displayed as a braided Drinfeld double in the sense of \cite{Maj2,Lau}, but unless $H$ is quasitriangular (and the coactions induced by the $R$-matrix), we need to view it over the base Hopf algebra $\Drin(H)$ instead. If the positive and negative part are perfectly paired, then we can give a formal power series describing the $R$-matrix and an appropriate subcategory (corresponding to the Drinfeld center) is braided. 

Particularly interesting examples of such braided Drinfeld doubles include the quantum groups $\Ug$ for generic $q$, and the small quantum groups $u_q(\mathfrak{g})$ (see \cite[Section~4]{Maj2}). Their construction uses the concept of a \emph{weak} quasitriangular structure for which a similar statement to Proposition \ref{recoverdouble} can be made. We will see in Section~\ref{multiparametersection} that multiparameter quantum groups can be viewed as examples of asymmetric braided Drinfeld doubles that are not symmetric. Further, all the pointed Hopf algebras classified in the main result of this paper (Theorem \ref{mainclassificationthm}), under the additional assumption that the braiding is of separable type and some commutators do not vanish, are asymmetric braided Drinfeld doubles (Theorem \ref{drinfeldtheorem}).


\section{Classification over a Group}\label{section2}

In this section, we denote by $A= T(V)\otimes kG\otimes T(V^*)$ a bialgebra with triangular decomposition over a group algebra $kG$. Note that we do not assume $G$ to be finite.

\subsection{Preliminary Observations}\label{preliminaryobs}

Hopf algebras that are generated by grouplike and skew-primitive elements are always pointed. We show that assuming a Hopf algebra has triangular decomposition over a group and is of what we call \emph{weakly separable type}, it is generated by skew-primitive elements and hence pointed.

\begin{lemma}
For a bialgebra $A$ with triangular decomposition over $kG$ as above, there exists a basis $v_1,\ldots,v_n$ of $V$ and $f_1, \ldots, f_n$ of $V^*$, as well as invertible matrices $M$ and $N$ such that
\begin{align}
\Delta(v_i)&=v_i\otimes g_i+\sum_j{M_{ji}h_j\otimes v'_j},&\Delta(f_i)&=f_i\otimes a_i+\sum_j{N_{ji}b_j\otimes f'_j}\label{coprodonv},
\end{align}
where $v'_1,\ldots,v'_n$ is another basis of $V$, and $f_1',\ldots,f_n'$ of $V^*$.
\end{lemma}
\begin{proof}
Let $v_1, \ldots, v_n$ be a homogeneous basis for the YD-compatible grading $\delta_r$ and $v'_1,\ldots,v'_n$ a homogeneous basis for $\delta_l$. The form (\ref{coprodonv}) of the coproducts is obtained by letting $M$ be the base change matrix from $\lbrace v_i\rbrace$ to $\lbrace v_i'\rbrace$. The same argument works for the dual $V^*$, denoting the base change matrix from $\lbrace f_i\rbrace$ to $\lbrace f_i'\rbrace$ by $N$.
\end{proof}

\begin{lemma}
A bialgebra $A$ with a triangular decomposition over $kG$ as above is a Hopf algebra, with antipode $S$ given on generators of the form $v_i$, $f_i$ as in (\ref{coprodonv}) by
\begin{align}
S(v_i)&=-\sum_{j}M_{ji}(h_j^{-1}\triangleright v'_j)h_j^{-1}g_i^{-1},&S(f_i)&=-\sum_{j}N_{ji}(f'_j\triangleleft b_j)b_j^{-1}a_i^{-1}.
\end{align}
\end{lemma}
\begin{proof}
The antipode axioms require that $S$ is of the form stated, using that $kG$ is a Hopf subalgebra, cf. (\ref{antipodecond1})--(\ref{antipodecond2}). As $T(V)$ and $T(V^*)$ are free, defining $S$ on the generators extends uniquely to an anti-algebra and anti-coalgebra map on all of $A$.
\end{proof}

\begin{definition}\label{indecprop}
A Hopf algebra $A$ with triangular decomposition over a group is called of \emph{weakly separable type} if the right degrees $g_i,\ldots, g_n$ of $V$ are pairwise distinct group elements, and  the same holds for the left degrees $h_1,\ldots,h_n$ of $V$ as well as the dual degrees.
\end{definition}

We observe that being of weakly separable type over a group implies that $V$ and $V^*$ have 1-dimensional homogeneous components. This gives that for a homogeneous basis element $v_i$ of degree $a_i$, $g\triangleright v_i\neq 0$ is homogeneous of degree $ga_ig^{-1}$ which hence has to be a scalar multiple of a basis element $v_{g(i)}$ where $g(i)$ is an index $1, \ldots, n$. Hence we obtain an action of $G$ on $\lbrace 1,\ldots, n\rbrace$. To fix notation, we write
\begin{align}
g\triangleright v_i&=\lambda_i(g)v_{g(i)},&f_i\triangleleft g=\mu_i(g)f_{g(i)}&.
\end{align}
We will see that for $A$ of weakly separable type, the base change matrices $M$, $N$ are diagonal matrices and can be chosen to be the identity matrix by rescaling of the diagonal bases. This implies that $A$ is generated by primitive and group-like elements and hence pointed. It is a conjecture in \cite[Introduction]{AS} that all finite-dimensional pointed Hopf algebras over an algebraically closed field of characteristic zero are in fact generated by skew-primitive and group-like elements.

\begin{proposition}\label{primitiveprop}
If $A$ is of weakly separable type, then there exist bases $\lbrace v_i\rbrace$ of $V$ and $\lbrace f_i\rbrace$ of $V^*$ consisting of skew-primitive elements, such that
\begin{align}\label{primitivecoprod}
\Delta(v_i)&=v_i\otimes g_i+h_i\otimes v_i,
&\Delta(f_i)&=f_i\otimes a_i+b_i\otimes f_i,
\end{align}
and the antipode on these skew-primitive elements is given by $S(v_i)=(h_i^{-1}\triangleright v_i)h_i^{-1}g_i^{-1}$, $S(f_i)=(f_i\triangleleft b_i)b_i^{-1}a_i^{-1}$.
\end{proposition}
\begin{proof}
Consider the right and left coactions $\delta_r$ and $\delta_l$ from Section \ref{definitions}.
Choosing a basis $v_1,\ldots,v_n$ homogeneous for $\delta_l$ and $v'_1,\ldots, v'_n$ homogeneous for $\delta_r$, (\ref{coprodonv}) gives
\begin{equation}
\Delta(v_i)=v_i\otimes g_i+\sum_j{M_{ji} h_j\otimes v'_j},
\end{equation}
where $M=(M_{ji})$ is the base change matrix.
By coassociativity, we find that
\begin{equation}\label{eq2}
\sum_{j,k}{M_{ji}(M^{-1})_{kj}h_j\otimes v_k\otimes g_k}=\sum_j{M_{ji} h_j\otimes v'_j\otimes g_i}.
\end{equation}
By weak separability of $\delta_r$ and $\delta_l$ we now have for each $j=1,\ldots, n$:
\begin{align}
\sum_{k}{M_{ji}(M^{-1})_{jk}v_k\otimes g_k}=M_{ji}v'_j\otimes g_i.
\end{align}
Note that $M_{ji}\neq 0$ for at least some $i$. This implies that $(M^{-1})_{kj}=0$ unless $k=i$ as the $g_i$ are all distinct. Further, if $M_{ji}\neq 0$, then $v_i$ and $v'_j$ are proportional. This can only be true for at most one $i$ for given index $j$ by weak separability. Hence by reordering the basis $v'_1, \ldots, v'_n$ we find that $M$ is a diagonal matrix and can rescale the basis $\lbrace v'_i\rbrace$ such that $M$ is the identity matrix. Hence we have
$\Delta(v_i)=v_i\otimes g_i + h_i\otimes v_i$.
The antipode conditions for $A$ give (using Lemma \ref{hopflemma}) that $S$ is of the form claimed.
\end{proof}

\begin{remark}
The bases $\lbrace v_i\rbrace$ and $\lbrace f_i \rbrace$ do not necessarily need to be orthogonal with respect to the pairing $\langle ~,~\rangle$.
We will see in Theorem \ref{mainclassificationthm} that if the characters $\lambda_i$ are all distinct, then the bases can be chosen to be dual bases.
\end{remark}

\begin{remark}\label{primitivenotation}
In the following, we fix a basis $v_1, \ldots,v_n$ for $V$ and $f_1,\ldots, f_n$ for $V^*$ such that
\begin{align}
\Delta(v_i)&=v_i\otimes g_i+h_i\otimes v_i,&\Delta(f_i)&=f_i\otimes a_i+b_i\otimes f_i, &i=1,\ldots,n.
\end{align}
\end{remark}

A direct observation from Proposition \ref{primitiveprop} is that the algebra $A$ is generated by primitive and grouplike elements (which precisely give the group $G$) and hence pointed. We have the following restrictions on the group structure.

\begin{proposition}\label{symmetricprop}
In the group $G$, the relations $[g_i,a_j]=[h_i,a_j]=1$ and $[h_i, b_j]=[g_i,b_j]=1$ hold  for all $i,j=1, \ldots,n$. In particular, if $A$ has a symmetric triangular decomposition, then the subgroup of $G$ generated by all degrees is abelian.

Further, the following identities for the characters of the group action hold:
\begin{align}\label{characteridentities}
\mu_j(h_i)&=\lambda_i(a_j)^{-1},&\mu_j(g_i)=\lambda_i(b_j)^{-1}.&
\end{align}
\end{proposition}
\begin{proof}
The commutator relations follow by applying (\ref{eqn3}) and (\ref{eqn4}) to a pair of homogeneous basis elements of $V$ and $V^*$ with respect to $\delta_l, \delta_r^*$ (or $\delta_r, \delta_l^*$). Then it follows from (\ref{eqn1}) and (\ref{eqn2}) that $h_i(j)=j$, $a_j(i)=i$, $g_i(j)=j$ and $b_j(i)=i$ by the PBW theorem. This implies the relations (\ref{characteridentities}). In the symmetric case, $a_i=g_i^{-1}$ and $b_i=h_i^{-1}$ which forces the subgroup generated by all degrees to be abelian.
\end{proof}


\subsection{Classification in the Free Case of Weakly Separable Type}\label{freeclassification}

We are now in the position to classify all Hopf algebras $A$ with triangular decomposition of weakly separable type (cf. Definition \ref{indecprop}). This will enable us to view the Hopf algebras arising from this classification as analogues of multiparameter quantum groups in Section~\ref{multiparametersection}. We start by considering the case $A=T(V)\otimes kG\otimes T(V^*)$ which is referred to as the \emph{free} case. 

\begin{proposition}\label{pointedthm}
For the Hopf algebra $A$ with triangular decomposition of weakly separable type to be indecomposable as a coalgebra it is necessary that $G$ is generated by elements $k_1,\ldots,k_n$, $l_1,\ldots, l_n$ such that there exist generators $v_i$ of $V$ and $f_i$ of $V^*$ which are skew-primitive of the form
\begin{align}\label{coproductform}
\Delta(v_i)&=v_i\otimes k_i+1\otimes v_i, &\Delta(f_i)=f_i\otimes 1+l_i\otimes f_i,
\end{align}
with $[k_i,l_j]=1$ for all $i,j$. For the characters of the actions on the homogeneous components of $V$ and $V^*$ we require that
\begin{equation}\label{characterrequirement}
\mu_j(k_i)=\lambda_i(l_j)^{-1}.
\end{equation}
\end{proposition}
\begin{proof}
To determine when pointed Hopf algebras are indecomposable as coalgebras, consider the graph $\Gamma_A$ described in \ref{indecomposability}. Assume that $A$ has generators given as in Remark~\ref{primitivenotation}. We claim that the connected components of $\Gamma_A$ are in bijection with the double cosets of the subgroup
\[
Z:=\langle g_1^{-1}h_1,\ldots, g_n^{-1}h_n, a_1^{-1}b_1,\ldots, a_n^{-1}b_n\rangle
\]
in $G$ which partition $G$. Indeed, using that the elements $gv_i$ and $gf_i$ are skew-primitive of type $(gg_i,gh_i)$ and $(ga_i,gb_i)$, we find that the connected component of $g$ contains, for $i=1,\ldots, n$, of the strands
\[
\ldots \longrightarrow g(g_i^{-1}h_i)^{-2}\longrightarrow g(g_i^{-1}h_i)^{-1}\longrightarrow g \longrightarrow g(g_i^{-1}h_i)^1\longrightarrow g(g_i^{-1}h_i)^{2} \longrightarrow \ldots 
\]
for $i=1,\ldots, n$ and the same strands with $a_i^{-1}b_{i}$ instead of $g_i^{-1}h_i$ (and with $g$ multiplied on the right).
Moreover, as the elements $gv_i$, $gf_i$, $v_ig$, $f_ig$ (and possibly linear combinations of products of them, which would again be skew-primitive with degrees given by elements in $Z$) are the only skew-primitive elements in $A$, and thus give the only arrows in $\Gamma_A$, two elements $g$ and $h$ are in the same connected component if and only if $z_1gz_2=z_3hz_4$, for some $z_i\in Z$.
Thus, $A$ is indecomposable if and only if $G$ equals the connected component of $1$ in the graph $\Gamma_A$, hence if $G=Z$ which is the group generated by the elements $k_i:=h_i^{-1}g_i$, $l_i:=a_i^{-1}b_i$ for $i=1,\ldots, n$. Thus, in order to obtain indecomposability, the coproducts are of the form as stated in (\ref{coproductform}). This is achieved by replacing the generators $v_i$ by $v_ih_i^{-1}$ and $f_i$ by $a_i^{-1}f_i$. The remaining statements follow directly from Proposition \ref{symmetricprop}.
\end{proof}

\begin{theorem}\label{mainclassificationthm}
For an indecomposable Hopf algebra $A$ of weakly separable type as in Proposition \ref{pointedthm}, the commutator relations (\ref{commrel}) are of the form
\begin{align}\label{commrel2}
[f_i,v_j]&=\gamma_{ij}(k_j-l_i),&\forall 1\leq i, j\leq n,
\end{align}
where $\gamma_{ij}$ are scalars in $k$ such that $\gamma_{ij}=0$ whenever $\lambda_i\neq \lambda_j$ in which case also $\langle f_i,v_j\rangle=0$, or if either of $l_i$ of $k_j$ are not central.
Conversely, any choice of such scalars gives a pointed Hopf algebra of this form.
\end{theorem}
\begin{proof}
With the work done in Proposition \ref{pointedthm}, it remains to verify that the form of the commutator relation (\ref{commrel}) is as stated. Recall that in \cite[3.1]{BB}, the commutator relation is given by means of a quasi-coaction. That is a morphism $\delta\colon V\to kG\otimes V$ satisfying (\ref{ydcond}) and (\ref{commrel}). Such a morphism has the general form
\begin{align}
\delta(v_j)=v_j^{[-1]}\otimes v_j^{[0]}&=\sum_{k,g}{\alpha_{k,g}^j g\otimes v_k},&\alpha_{k,g}^i\in k,
\end{align}
on the basis elements from (\ref{coproductform}). Then (\ref{eqn5}), which is required for $A$ to be a bialgebra, rewrites as
\begin{align*}
\sum_{k,g}{\alpha_{k,g}^j (g\otimes k_j+l_i\otimes g)\langle f_i,v_k \rangle}
&=\sum_{k,g}{\alpha_{k,g}^j g\otimes g\langle f_i,v_k \rangle}, &\forall i,j.
\end{align*}
For each $i$, there exists $k$ such that $\langle f_i,v_k\rangle \neq 0$. For given $i$, we denote the set of indices such that $\langle f_i,v_k\rangle \neq 0$ by $I_i$. For such $k\in I_i$, we find that
 $\alpha_{k,g}^j=0$ for $g\neq k_j, l_i$, and $\alpha_{k,k_j}^j=-\alpha_{k,l_i}^j$. Thus, we obtain that $\delta$ is of the form
\begin{equation}\label{qcoactionform}
\delta(v_j)=v_j^{[-1]}\otimes v_j^{[0]}=\sum_{i=1}^n{\gamma_{ij} (k_j-l_i)\otimes v'_i},
\end{equation}
where $\gamma_{ij}=\sum_{k\in I_i}{\alpha_{k,k_j}^j {\langle f_i,v_k\rangle\abs{I_i}}}$ and $\lbrace v'_i\rbrace$ is the dual basis of $V$ to $\lbrace f_i\rbrace$.
Conversely, given arbitrary scalars $\gamma_{ij}$ for $i,j=1,\ldots,n$, we can define a quasi-coaction by the same formula (\ref{qcoactionform}). Then $\delta$ is YD-compatible with the given action of $G$ on $V$ if and only if (cf. condition (A) in \cite[Theorem~A]{BB})
\begin{align*}
\gamma_{ij}\mu_i(g)(gk_j-gl_{i})=&g[f_i\triangleleft g,v_j]\stackrel{(\text{A})}{=}[f_i, g\triangleright v_j]g=\gamma_{ij}\lambda_j(g)(k_jg-l_ig).
\end{align*}
This implies $\lambda_j=\mu_i$ whenever $\gamma_{ij}\neq 0$. Further, if $\gamma_{ij}\neq 0$ we need $l_i, k_j\in Z(G)$. These two requirement ensure that $\delta$ is YD-compatible.

Further, by duality of the action, if $\langle f_i,v_j \rangle\neq 0$ then $\lambda_i=\mu_j$. As for given $i=1,\ldots, n$, $\langle f_i,v_j \rangle\neq 0$ for some $j$ we have that $\lambda_i=\mu_j$ for at least some $j$, and vice versa. Hence, the set of characters and dual characters are in bijection. We can change the numbering and assume without loss of generality (recall that we are in the weakly separable case) to obtain 
\begin{equation}
\lambda_i=\mu_i.
\end{equation}
From now on, we will hence only use the notation $\lambda_i$.
\end{proof}

\begin{example}
The most degenerate case, where $\gamma_{ij}=0$, gives the Hopf algebra $(T(V)\otimes T(V^*))\rtimes kG$ where the tensor algebras are again computed in the category of YD-modules over $kG$.
\end{example}

\begin{remark}
At this point, a comparison to \cite[2.4]{AS2} and \cite[4.3]{AS3} seems appropriate. The condition (\ref{commrel2}) is equivalent to the so-called \emph{linking relation} (\ref{asrel3}) after a change of generators $f_i\leftrightarrow l_i^{-1}f_i$, since in the form of Definition \ref{asform} all generators have coproducts $\delta(v_i)=v_i\otimes 1+g_i\otimes v_i$. Such a change of generators causes the commutators $\ad=[~,~]$ to become braided commutators $\underline{\ad}= \ide_{V^{\otimes 2}}-\Psi$. The scalars $\lambda_{ij}$ satisfy the condition (d) in \ref{classificationsurvey}, where for the characters $\chi_i\chi_j\neq \varepsilon$ implies $\lambda_{ij}=0$. This is the analogue of our condition $\lambda_i\neq \lambda_j$ implying $\gamma_{ij}=0$.

The linking relations also appear in the quantum group characterization of \cite[Theorem 4.3]{AS3}. Hence we can conclude that the classification in this section gives Hopf algebras with similar relations as appearing in the work of Andruskiewitsch and Schneider. The outcome here is more restrictive as in our setting relations of the form (\ref{asrel4}) cannot involve non-trivial elements in $kG$, and we also have a symmetry in the set $\chi$ of connected components due to the triangular decomposition.
\end{remark}

The situation where $\lbrace v_i\rbrace$ and $\lbrace f_i \rbrace$ are orthogonal bases deserves particular attention. In this case, the scalars $\gamma_{ij}=0$ for $i\neq j$. The following concept of separability ensure this.

\begin{definition}
Let $A$ have a triangular decomposition of weakly separable type over a group $G$. If the characters $\lambda_1, \ldots, \lambda_n$ are distinct for different indices, we will speak of a triangular decomposition of \emph{separable type}.

If $A$ is of the form as in Theorem \ref{mainclassificationthm}, we say that $A$  is \emph{non-degenerate} if $\gamma_{ii}\neq 0$ for all $i=1,\ldots,n$ (this implies $k_i\neq l_i$).
\end{definition}

Note that both definitions --- separability and non-degeneracy --- cause the group $G$ to be abelian, and hence the braidings on $V$ and $V^*$ to be of diagonal type.
Assuming non-degeneracy, we can adapt the terminology of \cite[5.5]{BB} that the braided doubles in this case come from \emph{mixed} YD-structures. A mixed YD-structure is a quasi-coaction $\delta$ that is a weighted sum $\sum{t_i\delta_i}$, where $\delta_i$ are YD-modules compatible with the same action, and $t_i$ are generic scalars. The quasi-YD-module in the theorem is the sum $\delta=\delta_r-(\delta_l^*)^*$, where $(\delta_l^*)^*$ is the YD-module given by $v_j\mapsto l_j\otimes v_j$, which is dual to $\delta_l^*$. We will see that in this case all the Hopf algebras arising are certain \emph{asymmetric} braided Drinfeld doubles (as defined in \ref{asymdrin}). In the symmetric case, these algebras are in fact braided Drinfeld doubles. In particular, their appropriately defined module categories (resembling the Drinfeld center) are braided.

\subsection{Interpretation as Asymmetric Braided Drinfeld Doubles}\label{quotientsection}

Assume in this section that $A$ is non-degenerate of indecomposable separable type over $G$. So far, we have only classified \emph{free} braided doubles over $kG$. That is, as a $k$-vector space $A\cong T(V)\otimes kG\otimes T(V^*)$ via the multiplication map. To capture examples such as quantum groups, it is necessary to consider quotients of $A$ by triangular ideals $J=( I,I^*)$ such that $A/J\cong T(V)/I\otimes kG\otimes T(V^*)/{I^*}$ is still a Hopf algebra (and thus pointed). Here $I\triangleleft T^{>1}(V)$ and $I^*\triangleleft T^{>1}(V^*)$ are ideals and also coideals, and $J\in \cI_{\Delta}(A)$. We will now refine our considerations from Section \ref{triangularhopfsect} to find for what ideals $I$ and $I^*$ this is the case. We will use the notation
\begin{align}
q_{ij}:=\lambda_j(k_i).
\end{align}
Then, by (\ref{characterrequirement}), we have that $\lambda_j(l_i)=q_{ji}^{-1}$, and the matrix $q=(q_{ij})$ describes the braiding on $V$ fully, i.e. it is of \emph{diagonal type}.

The collection of triangular Hopf ideals $\cI_\Delta(A)$ introduced in Section \ref{triangularhopfsect} can be described more concretely in the separable non-degenerate case: The ideals in $\cI_\Delta(A)$ are of the form $J=I\otimes kG\otimes T(V^*)+T(V)\otimes kG\otimes I^*$ where $I$ is an ideal in the collection $\cI_{(V,\delta_r)}$ for $V$ with the right coaction given by $\delta_r$, and $I^*$ is in $\cI_{(V^*,\delta_l^*)}$ for the left dual coaction $\delta_l^*$ on $V^*$. This follows using \cite[Proposition 5.10]{BB} and the description of triangular Hopf ideals in Lemma \ref{ideallemma}. We use that by (\ref{characterrequirement}) the braiding $\Psi_r$ coming from $\delta_r$ and $\Psi_l$ from $(\delta_l^*)^*$ on $V$ are given by
\begin{align}
\Psi_r(v_i\otimes v_j)&=q_{ij}v_j\otimes v_i,&&\Psi_l(v_i\otimes v_j)=q_{ji}^{-1}v_j\otimes v_i,.
\end{align}
That is $\Psi_l=\Psi_r^{-1}$, the inverse braiding. Thus, $I^*$ is just the dual $k$-vector space to $I$.

\begin{example}
In the quantum groups $A=\Ug$, the braiding satisfies the symmetry $q_{ij}=q^{i\cdot j}=q^{j\cdot i}=q_{ji}$ as the Cartan datum is symmetric. This implies that the relations in $I$ are symmetric under reversing the order of tensors $v_1\otimes \ldots\otimes v_n\leftrightarrow v_n\otimes \ldots\otimes v_1$. This can be verified explicitly by observing that in $\Ug$ the ideal $I$ is generated by $q$-Serre relations, which carry such a symmetry.
\end{example}

\begin{theorem}\label{drinfeldtheorem}
Let $A$ be an indecomposable bialgebra with triangular decomposition of separable non-degenerate type over $G$. Then $A$ is an asymmetric braided Drinfeld double.
\end{theorem}
That is, all quotients by triangular Hopf ideals $J\in \cI_\Delta(A)$ of algebras $A$ of separable non-degenerate type occurring in the classification of Theorem \ref{mainclassificationthm} are asymmetric braided Drinfeld doubles. If $J$ is maximal in $\cI_\Delta(A)$, then $A/J\cong U_{kG}(V, V^*)$.

\begin{proof}
Recall that every Hopf algebra with triangular decomposition is the quotient of a free braided double by a triangular Hopf ideal. 
We saw that in the free separable case the commutator relations are of the form $[f_i,v_j]=\delta_{ij}\gamma_{ii}(k_i-l_j)$. This is precisely the form of the asymmetric braided Drinfeld double of $V$ with right YD-module structure given by the right grading, and $V^*$ with left YD-module structure given by the left dual grading. The pairing is given by $\langle f_i,v_j\rangle=\delta_{ij}\gamma_{ii}$ here. We have to check that the braided Hopf algebras $T(V)$ and $T(V^*)$ of YD-modules over $G$ are dually paired when viewed in the category of left $kG$-modules. This however follows from condition (\ref{characterrequirement}). Taking the maximal quotient by a triangular ideal (or the left and right radical of the pairing) gives the asymmetric braided Drinfeld double $U_{kG}(V, V^*)$.
\end{proof}

If some of the parameters $\gamma_{ii}$ are zero, then the pointed Hopf algebras obtained are not asymmetric braided Drinfeld double any more (in the sense of Definition \ref{asymmetricdrinfelddef}).

\subsection{Recovering a Lie Algebra}\label{liealgebrasection}

We assume that $\operatorname{char} k=0$ in this section and study Hopf algebras with triangular decomposition of separable type which are of the form $U_{kG}(V, V^*)$ (see Theorem \ref{drinfeldtheorem}). The aim is to set the characters $\lambda_i$ and the group elements $k_i$, $l_i$ equal to 1. This way, we want to recover a Lie algebra $\fr{g}$ for any of the indecomposable pointed Hopf algebras of the form $U_{kG}(V, V^*)$, relating back to the question asked in the introduction of finding quantum groups for a given Lie algebra. The tool available for this is the Milnor--Moore theorem from \cite{MM} (see also \cite[Theorem 5.6.5]{Mon}) which shows that any cocommutative connected Hopf algebras is of the form $U(\fr{g})$ for a (possibly infinite-dimensional) Lie algebra $\fr{g}$.

There are technical problems with this naive approach. To set the elements $q_{ij}$ --- which will be replaced by formal parameters --- equal to one, we need to give an appropriate integral form to avoid that the modules collapse to zero. This rules out examples like e.g. $k[x]/(x^n)$ (and, more generally, the small quantum groups) which are braided Hopf algebras in the category of YD-modules over $k\mZ$, as here a generator of the group acts by a primitive $n$th root of unity $q$ on $x$, and $\mZ[q]\subset k$ is a cyclotomic ring.

As a first step, we introduce appropriate integral forms of $U_{kG}(V, V^*)$, for which we need the square roots of $q_{ij}$. 
We consider the subring $Z:=\mZ[q_{ij}^{\pm 1/2}]_{i,j}\subset k$ adjoining all square roots of the numbers $q_{ij}$ and their inverses. These will now be treated as formal parameters with certain relations between them, coming from the relations we have among them in $k$.

\begin{remark}
In this section, we assume that the ideal $\langle q_{ij}^{\pm 1/2}-1 \mid i,j=1,\ldots,n\rangle $ in $Z$ is a proper ideal, and hence $p\colon Z\to \mZ$, $q_{ij}^{\pm 1/2}\mapsto 1$ is a well-defined morphism of rings.
\end{remark}

This assumption is crucial in the formal limiting process. It, for example, prevents examples in which $q^n+q^{n-1}+\ldots+q+1=0$ as in cyclotomic rings. 

To produce an integral form, we replace a given YD-module $V$ over $kG$ of separable type as in the previous sections by a YD-module over $ZG$. For this, we can choose a $G$-homogeneous basis $v_1,\ldots,v_n$ and a homogeoneous dual basis $f_i,\ldots,f_n$ such that (possibly after rescaling)
\begin{align}
\langle f_i, v_j\rangle&=\frac{1}{q_{ii}^{1/2}-q_{ii}^{-1/2}}\delta_{ij}, &\forall i,j.
\end{align}
An important observation is that the Woronowicz symmetrizers, which are used to compute the Nichols ideal $I_{\op{max}}(V)$, have coefficients in $Z$. Hence their kernels will be $Z$-modules. That is, for $V^{\op{int}}$ defined as $Z\langle v_1, \ldots, v_n\rangle$, which is a YD-module over the group ring $ZG$, the Woronowicz symmetrizer $\Wor^n_{\op{int}}\Psi$ is a $Z$-linear map $V^{\op{int}\otimes n}\to V^{\op{int}\otimes n}$. Hence $I_{\op{max}}(V^{\op{int}}):=\ker \Wor_{\op{int}}\Psi$ is an ideal in $T(V^{\op{int}})$, the tensor algebra over $Z$.

In order to provide an integral form of $U_{kG}(V, V^*)$, we change the presentation by introducing new commuting generators, namely $[f_i,v_i]=:t_i$.
One verifies that the following commutator relations hold over $k$, as we are given the relation $t_i=\tfrac{1}{q_{ii}^{1/2}-q_{ii}^{-1/2}}(k_i-l_i)$ when working over the field:
\begin{align}
[f_i,t_j]&=\delta_{i,j}(q_{ii}^{1/2}k_if_i+q_{ii}^{-1/2}l_if_i),\label{tirel1}\\
[v_i,t_j]&=-\delta_{i,j}(q_{ii}^{-1/2}k_iv_i+q_{ii}^{1/2}l_iv_i).\label{tirel2}
\end{align}

\begin{definition}
The \emph{integral form} $U_{ZG}(V^{\op{int}}, V^{\op{int}*})$ of $U_{kG}(V, V^*)$ is defined as the graded Hopf algebra over the ring $Z$ generated by $v_1,\ldots, v_n$, of degree 1, $f_1,\ldots, f_n$ of degree $-1$, and the group elements $k_1, \ldots, k_n, l_1, \ldots, l_n\in G$, and additional elements  $t_1, \ldots, t_n$ of degree 0, subject to the relations of $I_{\op{max}}(V^{\op{int}})$ and $I_{\op{max}}^*(V^{\op{int}})$, bosonization relations
\begin{align}
gv_i&=(g\triangleright v_i)g,&f_ig&=g(f_i\triangleleft g),
\end{align}
as well as the relations (\ref{tirel1}), (\ref{tirel2}) and
\begin{align}
gv_i&=(g\triangleright v_i)g,\qquad f_ig=g(f_i\triangleleft g),\label{intrel1}\\
q_{ii}^{1/2}(k_i-l_i)&=(q_{ii}-1)t_i,\label{intrel2}\\
[f_i, v_j]&=\delta_{i,j}t_i,\label{intrel3}\\
[t_i,t_j]&=0.
\end{align}
The coproducts are given as before on the generators $f_i,v_i,k_i,l_i$ and $\Delta(t_i)=t_i\otimes k_i+l_i\otimes t_i$.
\end{definition}

Note that as $A=U_{ZG}(V^{\op{int}}, V^{\op{int}*})$ is a Hopf algebra over the commutative ring $Z$, the coproduct is a map $A\to A\otimes_Z A$. For the quantum groups $U_q(\fr{g})$ at generic parameter, the integral form in this case is the so-called \emph{non-restricted} integral form (see e.g. \cite[9.2]{CP}) which goes back to De Concini--Kac \cite{DK}. To set the parameters equal to one, and to consider extensions of Hopf algebras to fields, we use the following Lemma:

\begin{lemma}\label{hopflemma2}
Let $\phi \colon R\to S$ be a morphism of commutative algebras. We denote the category of Hopf algebras over $R$ by $\mathbf{Hopf}_{R}$. Then base change along $\phi$ induces a functor
\begin{align*}
\mathbf{Hopf}_{\phi}\colon \mathbf{Hopf}_{R}&\longrightarrow \mathbf{Hopf}_{S},&A&\longmapsto A\otimes_RS.
\end{align*}
\end{lemma}
\begin{proof}
Given a Hopf algebra $A$ which is an $R$-algebra, i.e. there is a morphism $R\to A$, we induce the multiplication and comultiplication on $A\otimes_RS$ using the isomorphism
\[
(A\otimes_RS)\otimes_S(A\otimes_RS)\cong (A\otimes_R A)\otimes_RS.
\]
It is easy to check that the Hopf algebra axioms are preserved under base change.
\end{proof}

\begin{proposition}
There is an isomorphism of graded Hopf algebras
\[
U_{ZG}(V^{\op{int}}, V^{\op{int}*})\otimes_Zk\stackrel{\sim}{\longrightarrow} U_{kG}(V, V^*).
\]
\end{proposition}
\begin{proof}
Recall that $Z\leq k$ by construction. Extending to $k$, we are able to divide by $q_{ii}-1$ in (\ref{intrel2}), and recover the original commutator and bosonization relations in $U_{kG}(V, V^*)$. It remains to verify that
\[
I_{\op{max}}(V^{\op{int}})\otimes_Zk=\ker \Wor_{\op{int}}\Psi\otimes_Z k=\ker \Wor \Psi=I_{\op{max}}(V). 
\]
This follows by noting that $k$ is flat as a $Z$-module (since the function field $K(Z)$ is flat over $Z$ as a localization, and $k$ is free over $K(Z)$), and $V^{\op{int}}\otimes_Zk\cong V$ as $k$-vector spaces.
\end{proof}

\begin{definition}
We define the \emph{classical limit} of $U_{kG}(V, V^*)$ as the algebra
\[
U_k^{\op{cl}}(V, V^*):=\bigslant{(U_{ZG}(V^{\op{int}}, V^{\op{int}*})\otimes_Z\mZ)\otimes_\mZ k}{( \ker \varepsilon_G)},
\]
using the morphism $p\colon Z\to \mZ$ mapping all $q_{ij}^{\pm 1/2}$ to $1$, and the two sided ideal $( \ker \varepsilon_G)$ generated by the kernel of the augmentation map $\varepsilon_G\colon kG\to k$ mapping all group elements to $1$. Note that this ideal is a Hopf ideal.
\end{definition}

That is, to obtain the classical limit, we first set the parameters $q_{ij}^{\pm 1/2}$ equal to 1 in the integral form and then extend the resulting $\mZ$-module to a $k$-vector space, and finally set the group elements equal to 1 along the counit $\varepsilon_G\colon kG\to k$. We obtain a primitively generated Hopf algebra, and hence a Lie algebra, this way:

\begin{proposition}
The classical limit $U_k^{\op{cl}}(V, V^*)$ is a connected Hopf algebra, generated by primitive elements. Hence, for the  Lie algebra $\fr{p}_V$ of primitive elements, $U(\fr{p}_V)=U_k^{\op{cl}}(V, V^*)$. This algebra is generated by triples $f_i,v_i,t_i$ which form a subalgebra isomorphic to $U(\fr{sl}_2)$.
\end{proposition}

\begin{proof}
Lemma \ref{hopflemma2} ensures that $U_k^{\op{cl}}(V, V^*)$ is a Hopf algebra over $k$, and freeness of $V^{\op{int}}$ over $Z$ ensures that the positive and negative part do not collapse to the zero space. In particular, the $k$-vector space $V^{\op{int}}\oplus V^{\op{int}*}$ embeds into the Lie algebra $\fr{p}_V$ of primitive elements.
In the classical limit, we obtain the relations
\begin{align}
[f_i,v_j]&=\delta_{i,j}t_i,
&[f_i,t_j]&=2\delta_{i,j}f_i,
&[v_i,t_j]&=-2\delta_{i,j}v_i.
\end{align}
Hence every triple $f_i, v_i, t_i$ generates a Lie subalgebra of $\fr{p}_V$ isomorphic to $\fr{sl}_2$. Note that $U_{k}^{\op{cl}}(V, V^*)$ is generated by the $t_i$ and the primitive elements:
\begin{align*}
\Delta(f_i)=&f_i\otimes 1+1\otimes f_i,&\Delta(v_i)&=v_i\otimes 1+1\otimes v_i.
\end{align*}
We also compute
\[
\Delta(t_i)=\Delta([f_i,v_i])=[f_i,v_i]\otimes k_i+l_i\otimes[f_i,v_i]=t_i\otimes k_i+l_i\otimes t_i.
\]
Hence, $t_i$ is skew-primitive in $U_{ZG}^{\op{int}}(V, V^*)$ and primitive in the classical limit. Thus, $U^{\op{cl}}_{k}(V,V^*)$ is a pointed Hopf algebra over the trivial group. That is, a \emph{connected} pointed Hopf algebra. It is further cocommutative and Theorem 5.6.5 in \cite{Mon} implies that such a Hopf algebra is of the form $U(\fr{g})$ where $\fr{g}=\fr{p}_V$ in $\operatorname{char}k=0$.
\end{proof}

Note that $U_{k}^{\op{cl}}(V, V^*)$ is a braided double over the polynomial ring $S(T)$, where $T=k\langle t_1,\ldots,t_n \rangle$ (which is not necessarily $n$-dimensional). The action is given by $t_j\triangleright v_i=2\delta_{i,j}v_i$, and the quasi-coaction is given by $\delta(v_i)=t_i\otimes v_i$ which is \emph{not} a coaction, hence $U_{ZG}^{\op{int}}(V, V^*)$ is \emph{not} a braided Heisenberg double. It is also not an asymmetric braided Drinfeld double.

\begin{example}
For $U_q(\fr{g})$, $\fr{g}$ a semisimple Lie algebra, viewed as a braided Drinfeld double, the classical limit is $U(\fr{g})$.
\end{example}

We can also compute examples that do not give finite-dimensional semisimple Lie algebras. As a general rule, the relations between the parameters $q_{ij}$ determine the relations in the Lie algebra. It is easy to construct free examples, for which there are no relations between the $v_1,\ldots, v_n$ by choosing algebraically independent parameters $q_{ij}$. The work of \cite{Ros} and \cite{AS3} give restrictions on examples satisfying the growth condition of finite Gelfand--Kirillov dimension. We will view their results in the setting of this paper in Section \ref{section3}.


\section{Classes of Quantum Groups}\label{multiparametersection}

In this section, we relate the classification from Section~\ref{section2} to various classes of examples which are often regarded as quantum groups. This includes the multiparameter quantum groups studied by \cite{FRT, Res, Sud,AST} and others in Section \ref{quantumgroupsmulti}, a characterization of Drinfeld--Jimbo quantum groups in Section \ref{section3}, and classes of examples of pointed Hopf algebras from the work of Radford in Section \ref{radford}. The classification in Theorem \ref{mainclassificationthm} points out natural generalizations of these classes of examples\footnote{While this paper was under revision, it was pointed out by Dr Gast\'on Andr\'es Garc\'ia that a further series of examples of asymmetric braided Drinfeld doubles is given in \cite[Definition~7]{HPR} and described in \cite{Gar} using a family of pointed Hopf algebras defined in \cite{ARS}.}. We finally sketch how one can define analogues of quantum groups using triangular decompositions over other Hopf algebras than $kG$.

\subsection{Multiparameter Quantum Groups}\label{quantumgroupsmulti}

Let $k$ be a field of characteristic zero. For the purpose of this section, let $\lambda \in k$ be generic, and $p_{ij}\in k$ for $1\leq i<j\leq n$. Assume that $p_{ii}=1$ and $p_{ji}=p_{ij}^{-1}$. Following \cite{AST,CM} and to fix notation, we set

\begin{align*}
&\kappa_j^{(i)}=\begin{cases}p_{ij},& \text{if }i<j,\\ \lambda, & \text{if } i=j,\\
\tfrac{\lambda}{p_{ji}}, 
& \text{if }i>j.\end{cases}&&\lambda_j^{(i)}=\begin{cases}\tfrac{\lambda}{p_{ij}},& \text{if }i<j,\\ \lambda, & \text{if } i=j,\\
p_{ji}, & \text{if }i>j.
\end{cases}
\end{align*}
We will provide a variation of the presentation of \cite{AST,CM} in order to display the multiparameter quantum group $U_{\lambda,\underline{p}}(\mathfrak{gl}_n)$ as a Hopf algebra with triangular decomposition.

\begin{example}[Multiparameter quantum groups]
We define on  $F=k \langle f_1, \ldots, f_{n-1}\rangle$ a YD-module structure over an abelian group $G$ with generators $k_1, \ldots, k_{n-1}$, $l_1, \ldots, l_{n-1}$. Denote the dual by $E=k\langle e_1,\ldots, e_{n-1}\rangle$, where the pairing is given by $\langle e_i,f_j\rangle=(1-\lambda)\delta_{ij}$.  The YD-structure is of separable type, and given by assigning the right degree $k_i$ to $f_i$, and the left degree $l_i$ to $e_i$, and actions
\begin{align}
k_i\triangleright f_j&=\lambda_j(k_i)f_j=\frac{\lambda_{j+1}^{(i)}\lambda_{j}^{(i+1)}}{\lambda_{j}^{(i)}\lambda_{j+1}^{(i+1)}}f_j,\\
l_i\triangleright f_j&=\lambda_j(l_i)f_j=\frac{\kappa_j^{(i)}\kappa_{j+1}^{(i+1)}}{\kappa_{j+1}^{(i)}\kappa_j^{(i+1)}}f_j,
\end{align}
for $i,j=1,\ldots n-1$. We will relate the \emph{multiparameter quantum group} $U_{\lambda,\underline{p}}(\mathfrak{gl}_n)$ to be the asymmetric braided Drinfeld double $U_{kG}(F,E)$.
\end{example}

Note that the definition of $U_{kG}(F,E)$ is possible as (\ref{characterrequirement}) holds, i.e.
\begin{align*}
q_{ij}:=\lambda_j(k_i)=\frac{\lambda_{j+1}^{(i)}\lambda_{j}^{(i+1)}}{\lambda_{j}^{(i)}\lambda_{j+1}^{(i+1)}}=\frac{\kappa_{i+1}^{(j)}\kappa_i^{(j+1)}}{\kappa_i^{(j)}\kappa_{i+1}^{(j+1)}}=\lambda_i(l_j)^{-1}.
\end{align*}
The commutator relation in $U_{kG}(F,E)$ is given by
\begin{equation}
[e_i,f_j]=(1-\lambda)\delta_{ij}(k_i-l_i).
\end{equation}
The following isomorphism displays $U_{kG}(F,E)$ as an indecomposable subalgebra of a multiparameter quantum group considered in the literature:

\begin{proposition}
There is an isomorphism of Hopf algebras $U_{kG}(F,E)\cong U'$ where $U'$ is a Hopf subalgebra of the multiparameter quantum group $U=U_{\lambda,\underline{p}}(\fr{gl}_n)$. \end{proposition}

\begin{proof}
We prove the theorem by first considering the morphism
\[
\phi\colon T(E)\otimes kG\otimes T(F)\longrightarrow U.
\]
Such a morphism will descent to an injective morphism $\overline{\phi}\colon U_{kG}(F,E)\to U$ by the following Lemma \ref{qserrelemma}. We further note that the image $\op{Im}{\overline{\phi}}=:U'$ is a Hopf subalgebra isomorphic to $U_{kG}(F,E)$.
Denote the generators of $U$ by $E_i,F_i$ for $i=1,\ldots,n-1$ and group elements $K_i,L_i$ for $i=1,\ldots, n$ (see \cite[4.8]{CM}). The map $\phi$ is defined by
$\phi(e_i)=\lambda E_iK^{-1}_{i+1}K_i$, $\phi(f_i):=F_i$, $\phi(k_i)=L_{i+1}L_i^{-1}$, and $\phi(l_i):=K_{i+1}^{-1}K_i$. One checks directly that the relations in the free braided double $T(E)\otimes kG\otimes T(F)$ are preserved under this map, using the presentation in \cite[4.8]{CM} for $U$.
\end{proof}

\begin{lemma}\label{qserrelemma}
The largest ideal in $\cI_\Delta(A)$ for $A=U_{kG}(F,E)$ is generated by the quantum Serre relations 
\begin{align}
\underline{\ad}(e_i)^{1-a_{ij}}(e_j)=\underline{\ad}(f_i)^{1-a_{ij}}(f_j)=0,
\end{align}
where $\underline{\ad}(e_i)(e_j)=e_ie_j-q_{ij}e_je_i$.
\end{lemma}
\begin{proof}
It follows from Lemma \ref{ideallemma} that the maximal ideal $J$ in $\cI_\Delta(A)$ is given by $J=( I,I^*)$ where $I$ is the Nichols ideal of the YD-module $F$. 

Generation of the maximal triangular ideal by quantum Serre relations for $U_{\lambda,\underline{p}}(\fr{gl}_n)$ follows from Lemma 4.5 in \cite{CM}.
For this, it is crucial that $\lambda$ is not a root of unity. The proof uses the observation in \cite{Res}, or \cite{AST} for the deformed function algebra, that multiparameter quantum groups, using quantum coordinate rings, can be obtained via 2-cocycles from one-parameter quantum groups.
The fact that the quantum Serre relations generate the Nichols ideal then follows from Theorem 4.4 in \cite{CM} where it is shown that these relations generate the radical of a Hopf pairing. Using the map $\phi$, this result describes the Nichols ideals of $T(F)$, $T(E)$ as generated by quantum Serre relations.
\end{proof}

The result that the multiparameter quantum group $U_{\lambda,\underline{p}}(\mathfrak{gl}_n)$ is the asymmetric braided Drinfeld double $U_{kG}(F,E)$ can be seen as a generalization of the result in \cite{BW} where the two-parameter quantum groups were shown to be Drinfeld doubles.

\subsection{Characterization of Drinfeld--Jimbo Quantum Groups}\label{section3}

Let $\op{char} k=0$ in this section. In Section~\ref{section2} we observed that for an algebra $A$ with symmetric triangular decomposition of separable type to be an indecomposable pointed Hopf algebra, $G(A)$ needs to be abelian acting on $V$ by scalars. That means, in the terminology of \cite{AS} that the YD-braiding $\Psi(v\otimes w)=v^{(-1)}\triangleright w\otimes v^{(0)}$ is of \emph{diagonal type}, i.e. there exist non-zero scalars $q_{ij}$ such that $\Psi(v_i\otimes v_j)=q_{ij}v_j\otimes v_i$ for a basis $\{v_1,\ldots,v_n\}$. 

We fix a choice of YD-module structure over an abelian group $G$ for this section to describe the diagonal braiding. That is, $q_{ij}=\lambda_j(k_i)$ for the characters $\lambda_i$ by which $G$ acts on $kv_i$ and group elements $k_i$ such that $\delta(v_i)=k_i\otimes v_i$. It is a basic observation that the braided Hopf algebras $T(V)/I$ for $I\in \cI_V$, including the Nichols algebras for $V$, only depend on the braiding on $V$ (rather than the concrete choice of $\lambda_i$, $k_i$). However, different diagonal braidings $(V, \Psi)$ and $(V, \Psi')$ may give isomorphic braided Hopf algebras $T(V)/I$. Such isomorphisms can be obtained using the notion of \emph{twist equivalence} for diagonal braidings (which is a special case of the more general concept of twisting a Hopf algebra by a 2-cocycle).

\begin{definition}
Two braided $k$-vector spaces of diagonal type $(V,\Psi)$, $(V',\Psi')$ (given by scalars $q_{ij}$, $q_{ij}'$) are \emph{twist equivalent} if $V\cong V'$, $q_{ii}=q_{ii}'$, and $q_{ij}q_{ji}=q_{ij}'q_{ji}'$.
\end{definition}

\begin{lemma}\label{twistlemma}
If $(V,\Psi)$, $(V',\Psi')$ are twist equivalent of diagonal type, then $T(V)\cong T(V')$ as braided Hopf algebras in the category of braided $k$-vector spaces, preserving the natural grading.
\end{lemma}
\begin{proof}
For a proof see e.g. \cite[3.9--3.10]{AS}. We can find generators $v_i$ of $V$ and $v_i'$ of $V'$ such that the isomorphism $\phi$ is determined by $v_i\mapsto v_i'$. Defining a 2-cocycle $\sigma$ by $\sigma(v_i\otimes v_j)=q_{ij}'q_{ij}^{-1}$ for $i<j$ and $1$ otherwise, we find that the product $v_iv_j$ maps to the product twisted by $\sigma$.  
Note that the isomorphism is \emph{not} an isomorphism in the category of YD-modules over $kG$ unless $(V',\Psi')=(V,\Psi)$.
\end{proof}

For an ideal $I\in \cI_V$, denote the corresponding ideal under the isomorphism $T(V)\cong T(V')$ from Lemma~\ref{twistlemma} by $I'$. Then we conclude that $T(V)/I\cong T(V')/{I'}$ is also an isomorphism of braided Hopf algebras. In particular, $\cB(V)\cong \cB(V')$ for the corresponding Nichols algebras.

\begin{lemma}\label{twistdrinfelddoubles}
If $(V,\Psi)$ and $(V',\Psi')$ are twist equivalent, such that
\[
G=\langle k_1,\ldots, k_n \rangle\cong\langle k_1',\ldots, k_n' \rangle=G'
\]
via $k_i\mapsto k_i'$, then $U_{kG}(V, V^*)\cong U_{kG'}(V', {V'}^{*})$ as Hopf algebras.
\end{lemma}
\begin{proof}
By Lemma~\ref{twistlemma}, $T(V)/I\cong T(V')/{I'}$ and $T(V^*)/{I^*}\cong T({V'}^{*})/{{I'}^{*}}$. By the assumptions on the group generators, $k_i\mapsto k_i'$ extends to an isomorphism $kG\cong kG'$. Thus we can define a morphism $U_{kG}(V, V^*)\to U_{kG}(V', V'^{*})$ which is an isomorphism of $k$-vector spaces. Further, preservation of the bosonization condition can be checked on generators using the isomorphism $\phi$ from Lemma~\ref{twistlemma}. Finally, the commutator relation (\ref{commrel2}) is preserved using the isomorphism on $kG$.
\end{proof}

Diagonal braidings are a very general class of braidings. Quantized enveloping algebras at generic parameters however are based on braidings of specific type, called \emph{Drinfeld--Jimbo type}. Following \cite{AS3}, there are different classes of braidings which we distinguish:

\begin{definition}[{\cite[Definition~1.1]{AS3}}]
Let $(q_{ij})$ be the $n\times n$-matrix of a braiding of diagonal type.
\begin{enumerate}
\item[(a)] The braiding given by $(q_{ij})$ is \emph{generic} if $q_{ii}$ is not a root of unity for any $i=1,\ldots,n$.
\item[(b)] In the case $k=\mC$ we say the braiding $(q_{ij})$ is \emph{positive} if it is generic and all diagonal elements $q_{ii}$ are positive real numbers.
\item[(c)]
The braiding $(q_{ij})$ is of \emph{Cartan type} if $q_{ii}\neq 1$ for all $i$ and there exists a $\mZ$-valued $n\times n$-matrix $(a_{ij})$ with values $a_{ii}=2$ on the diagonal and $0\leq -a_{ij}<\ord q_{ii}$ for $i\neq j$, such that
\begin{equation}
q_{ij}q_{ji}=q_{ii}^{a_{ij}}\qquad \text{ for all }i,j.
\end{equation}
This implies that $(a_{ij})$ is a generalized Cartan matrix which may have several connected components. We denote the collection of these by $\chi$.
\item[(d)]
The braiding $(q_{ij})$ is of \emph{Drinfeld--Jimbo type (DJ-type)} if it is of Cartan type and there exist positive integers $d_1,\ldots, d_n$ such that for all $i,j$, $d_i a_{ij}=d_j a_{ji}$ (hence the matrix $(a_{ij})$ is symmetrizable), and for any $J\in \chi$, there exists a scalar $q_J\neq 0$ in $k$ such that $q_{ij}=q_J^{d_ia_{ij}}$ for any $i\in I$, and $j=1,\ldots, n$.
\end{enumerate}
\end{definition}

Some observations can be made about the Nichols algebras associated to braid\-ed vector spaces of DJ-type. First, observe that for a braiding of Cartan type with connected components $I_1,\ldots,I_n\in \chi$, we have that $\cB(V)$ is the braided tensor product $\cB(V_{I_1})\otimes \ldots\otimes \cB(V_{I_n})$ (\cite[Lemma 4.2]{AS4}).
Further, for $V$ with braiding $(q_{ij})$ of DJ-type where $q_{ii}$ are generic, the Nichols algebra can be computed explicitly by the quantum Serre relations (\cite[Theorem 15]{Ros}):
\[
\cB(V)=k\langle x_1,\ldots,x_n \mid \un{\ad}(x_i)^{1-a_{ij}}(x_j)=0, \forall i\neq j\rangle.
\]

We now bring the growth condition of finite \emph{Gelfand--Kirillov dimension} (GK-dimension) into the picture, using characterization results of \cite{Ros} of Nichols algebras with this property.

\begin{lemma}[\cite{Ros}]\label{rossolemma} Let $k=\mC$ and $(q_{ij})$ be the matrix of a braiding of diagonal type which is \emph{generic} such that the Nichols algebra $\cB(V)$ has finite Gelfand--Kirillov dimension. Then $(q_{ij})$ is of Cartan type.

Moreover, if the braiding is positive then it is twist equivalent to a braiding of DJ-type with finite Cartan matrix if and only if the GK-dimension is finite.
\end{lemma}
\begin{proof}
See \cite{AS3}, Corollary 2.12 and Theorem 2.13.
\end{proof}

\begin{corollary}
Let $A=U_{\mC G}(V,V^*)$, for $V$ of separable type, with generic positive braiding $(q_{ij})$. Then the following are equivalent
\begin{itemize}
\item[(i)] $A\cong U_q(\fr{g})$ for $\fr{g}$ a semisimple Lie algebra.
\item[(ii)] The braided $\mC$-vector space $V$ with braiding $(q_{ij})$ is twist equivalent to a braiding of DJ-type with Cartan matrix of finite type.
\item[(iii)] $\cB(V)$ has finite Gelfand--Kirillov dimension.
\item[(iv)] $A$ has finite Gelfand--Kirillov dimension.
\end{itemize} 
\end{corollary}
\begin{proof}
The equivalence of (ii) and (iii) is the statement of Lemma \ref{rossolemma} due to \cite{Ros}. Using Lemma \ref{twistdrinfelddoubles} we find that (ii) implies (i), while it is clear that (i) implies (ii). In fact, the GK-dimension of $\cB(V)$ for $V$ of DJ-type equals the number of positive roots \cite[2.10(ii)]{AS3}. Further, we observed that $A$ is of the form $U(\cD)$ in \cite[Theorem 4.3]{AS3} in Theorem \ref{drinfeldtheorem} provided that $V$ has finite Cartan type. This observation (together with Lemma~\ref{twistdrinfelddoubles}) gives that (ii) is equivalent to (iv) using Theorem 5.2 in \cite{AS3}.
\end{proof}

\begin{corollary}
The only indecomposable bialgebras with a symmetric triangular decomposition on $\cB(V)\otimes k\mZ^n\otimes \cB(V^*)$ of separable type, such that $V=\mC\langle v_1,\ldots,v_n\rangle$ is of positive diagonal type, and that no $v_i$ commutes with all of $V^*$ are isomorphic to $\Ug$ for a semisimple Lie algebra $\mathfrak{g}$.
\end{corollary}
\begin{proof}
This follows from the classification in Theorem \ref{mainclassificationthm}, combined with the  results of Rosso.  The Lie algebra $\fr{g}$ is determined by the Cartan matrix one obtains under twist equivalence in Lemma \ref{rossolemma}.
The technical condition that no $v_i$ commutes with all of $V^*$ ensures that $[f_i,v_i]\neq 0$ for a dual basis $f_1,\ldots,f_n$ of $V^*$, resembling the non-degeneracy condition that the scalars $\gamma_{ii}\neq 0$ in Theorem \ref{drinfeldtheorem}.
\end{proof}

This is a characterization for quantum groups at generic parameters. The work surveyed in \cite{AS,AS2} on finite-dimensional pointed Hopf algebras can be viewed as a characterization of small quantum groups.
The triangular decomposition can be interpreted as the case where the graph $\Gamma$ described in \ref{classificationsurvey} has two connected components, such that the corresponding generators for the two components give dually paired braided Hopf algebras.

The characterization suggests that if we are looking for examples outside of DJ-type, we have to consider braidings of generic Cartan type which are not positive. In fact, \cite[2.6]{AS3} gives an example that is generic of Cartan type, but not of DJ-type. We compute the associated double here:

\begin{example}
Let $G=\langle k_1,k_2\rangle\cong C_\infty\times C_\infty$ be a free abelian group with two generators. We define a two-dimensional YD-module $V$ over $G$ on generators $v_1$ of degree $k_1$, $v_2$ of degree $k_2$ via
\begin{align*}
k_1\triangleright v_1&=q v_1,&k_1\triangleright v_2&=q^{-1}v_2,&k_2\triangleright v_1&=q^{-1}v_1,&k_2\triangleright v_2&=-qv_2.
\end{align*}
Lemma 2.1 in \cite{AS3} shows that
\[
\cB(V)=\langle v_1,v_2\mid \un{\ad}(v_1)^3(v_2)=\un{\ad}(v_2)^3(v_1)=0\rangle.
\]
The asymmetric braided Drinfeld double $U_{\mC G}(V,V^*)$ is in fact a braided Drinfeld double if we define $V^*$ to be the dual YD-module. It is the Hopf algebra given on $\cB(V)\otimes \mC G\otimes \cB(V^*)$, subject to the relations
\begin{align*}
[f_1,v_i]&=\delta_{1,i}\frac{k_1-k_1^{-1}}{q^{1/2}-q^{-1/2}}, &[f_2,v_i]&=\delta_{2,i}\frac{k_2-k_2^{-1}}{iq^{1/2}+iq^{-1/2}},&\\
k_1v_2&=q^{-1}v_2k_1,&k_2v_1&=q^{-1}v_1k_2,\\
k_1v_1&=qv_1k_1,&k_2v_2&=-qv_2k_2,\\
k_1f_2&=qf_2k_1,&k_2f_1&=qf_1k_2,\\
k_1f_1&=q^{-1}f_1k_1,&k_2f_2&=-q^{-1}f_2k_2,
\end{align*}
and with coproducts
\begin{align*}
\Delta(v_i)&=v_i\otimes k_i+1\otimes v_i,& \Delta(f_i)&=f_i\otimes 1+k_i^{-1}\otimes f_i.
\end{align*}
\end{example}

Apart from such examples, we can also include examples where free and nilpotent generators are combined, hence capturing features of both small and generic quantum  groups. Here is such an example of small rank:

\begin{example}
Let $G=C_\infty\times C_p=\langle g_{\infty}\rangle \times\langle g_p\rangle$ the product of an infinite cyclic group and one of order $p$. We define a 2-dimensional YD-module over $G$ on $\mC v_\infty\oplus \mC v_p$, where $v_\infty$ has degree $g_\infty$, and $v_p$ has degree $g_p$. The group action is given by
\begin{align*}
g_p\triangleright v_p&=\xi_pv_p, &g_p\triangleright v_\infty&=\eta_p v_\infty,\\
g_\infty\triangleright v_p&=\eta_p^{-1}v_p, &g_\infty \triangleright v_\infty&=\eta_\infty v_\infty,
\end{align*}
where scalars with a subscript $p$ are primitive $p$th roots of unity, and $\eta_\infty$ is generic. We can now compute the Nichols algebra with generators $v_p$ and $v_\infty$. It is given by
\[
\cB(V)=\mC\langle v_p,v_\infty \rangle /{(v_p^p, v_pv_\infty-\eta_pv_\infty v_p )}.
\]
 We denote the dual YD-module by $V^*$ with generators $f_p$, $f_\infty$.
 
 The braided Drinfeld double on $\cB(V)\otimes k(C_p\times C_\infty)\otimes \cB(V^*)$ of the braided Hopf algebra $\cB(V)$ is a quantum group that combines both $u_q(\fr{sl}_2)$ and $U_q(\fr{sl}_2)$:

\begin{align*}
[f_p,v_i]&=\delta_{i,p}\frac{g_p-g_p^{-1}}{\xi_p^{1/2}-\xi_p^{-1/2}}, &[f_\infty,v_i]&=\delta_{\infty,i}\frac{g_\infty-g_\infty^{-1}}{\eta_\infty^{1/2}-\eta_\infty^{-1/2}},\\
g_pv_p&=\xi_pv_pg_p,&g_pv_\infty&=\eta_p v_\infty g_p,\\
g_\infty v_p&=\eta_p^{-1} v_pg_\infty,&g_\infty v_\infty&=\eta_\infty v_\infty g_\infty,\\
g_pf_p&=\xi_p^{-1}f_pg_p,&g_pf_\infty&=\eta_p^{-1}f_\infty g_p,\\
g_\infty f_p&=\eta_qf_pg_\infty,&g_\infty f_\infty&=\eta_\infty^{-1}f_\infty g_\infty.&
\end{align*}
and with coproducts
\begin{align*}
\Delta(v_i)&=v_i\otimes g_i+1\otimes v_i,& \Delta(f_i)&=f_i\otimes 1+g_i^{-1}\otimes f_i, &\text{for }i=p,\infty.
\end{align*}
Choosing instead $g_\infty\triangleright v_p=\xi_\infty v_p$ we obtain more examples where the Nichols algebra will involve other relations depending on choice of $\xi_\infty$.
\end{example}

\subsection{Classes of Pointed Hopf Algebras by Radford}\label{radford}

In \cite{Rad}, a class of pointed Hopf algebras $U_{(N,\nu, \omega)}$ was introduced (see also \cite{Gel} for generalizations). These Hopf algebras are associated to the datum of a positive integer $N$ and $1\leq \nu <N$ such that $N$ does not divide $\nu^2$, and $\omega\in k$ is a primitive $N$th root of unity in a field $k$. Denote $q:=\omega^\nu$ and $r=\abs{q^\nu}=\abs{\omega^{\nu^2}}$. We let $C_N$ denote a cyclic group of order $N$ generated by an element $a$.

The algebra $U_{(N,\nu,\omega)}$ is the braided Drinfeld double of the YD-module Hopf algebra $U_+:=k[x]/(x^r)$ over $C_p$, with grading given by $x\mapsto a^{\nu}\otimes x$ and action $a\triangleright x=q^{-1} x$. Note that $U_+$ is the Nichols algebra of the one-dimensional YD-module $kx$. The coalgebra structure is given by $\Delta(x)=x\otimes a^{\nu}+1\otimes x$, and $\Delta(y)=y\otimes 1+a^{-\nu}\otimes y$ for the dual generator $y$. Note further that the other Hopf algebra $H_{(N,\nu,\omega)}$ introduced by Radford is simply the bosonization $U_+\rtimes kC_N$ in this set-up. The algebras $U_{(N,\nu, \omega)}$ and $H_{(N,\nu,\omega)}$ are not indecomposable unless $\nu=1$. To obtain indecomposable pointed Hopf algebras, we can consider the subalgebras generated by $x, y$ and $a^{\nu}$ (respectively, $x$ and $a^{\nu}$). Since these only depend on the choices of $r$ and $q$ we denote these Hopf algebras by $U_{(r,q)}$ (respectively, $H_{(r,q)}$). Note that $U_{(r,1,q)}=U_{(r,q)}$.

\subsection{Quantum Group Analogues in Other Contexts}\label{conclusion}

To conclude this paper, we would like to adapt the point of view that quantum groups can also be studied over other Hopf algebras $H$ than the group algebra. For this, one can, motivated by the results of this paper, look for Hopf algebras $A$ with triangular decomposition over $H$. The property over a group that $A$ is of separable type can be generalized by requiring that the YD-modules $V$ with respect to the left and right coactions $\delta_r$ and $\delta_l$ are a direct sum of distinct (one-dimensional) simples. One-dimensionality of the simples is however a strong restriction.

As a first example, we can consider the case where $H$ itself is primitively generated, i.e. $H=k[x_1, \ldots, x_n]$ over a field of characteristic zero. If $A$ is a bialgebra with triangular decomposition over $H$, then  for $v\in V$, $\Delta(v)\in V\otimes H+H\otimes V$ implies that $\Delta(v)$ in fact equals $v\otimes 1+1\otimes v$ using the counitary condition. This gives that $A$ is generated by primitive elements and hence is a pointed Hopf algebra that is connected (i.e. the group-like elements are the trivial group). Now $A$ is in particular cocommutative, so Theorem 5.6.5 in \cite{Mon} implies (for $\operatorname{char} k=0$) that $A=U(\fr{g})$ where $\fr{g}$ is the Lie algebra of primitive elements in $A$. From this point of view, all quantum groups over $H=k[x_1,\ldots, x_n]$ are simply the classical universal enveloping algebras. Investigating bialgebras with triangular decomposition over other Hopf algebras $H$ can be the subject of future research.


\begin{acknowledgements}
A preliminary version of this paper is part of the PhD thesis of the author completed at the University of Oxford. I am grateful to my PhD advisor Prof Kobi Kremnizer for his guidance. I would also like to thank Dr Yuri Bazlov, Prof Arkady Berenstein, Prof Dan Ciubotaru and Prof Shahn Majid for helpful discussions on the subject matter.

\end{acknowledgements}


\bibliography{biblio}
\bibliographystyle{spmpsci}

\end{document}